\documentclass[letterpaper]{article}
\usepackage{amsthm}
\usepackage{amssymb,latexsym,graphics,enumerate}
\usepackage[mathscr]{eucal}
\usepackage{amsmath,amsfonts,amsthm,amssymb}
\usepackage{color}
\usepackage{hyperref}
\usepackage[left=1in,right=1in,top=1in,bottom=1in]{geometry}

\numberwithin{equation}{section}

\newcommand{\rr}{\mathbb{R}}

\newcommand{\lan}{\langle}
\newcommand{\ran}{\rangle}
\newcommand{\be}{\begin{eqnarray*}}
\newcommand{\bel}{\begin{eqnarray}}
\newcommand{\ee}{\end{eqnarray*}}
\newcommand{\eel}{\end{eqnarray}}
\newcommand{\ba}{\begin{aligned}}
\newcommand{\ea}{\end{aligned}}
\newcommand{\de}{\Delta}
\newcommand{\al}{\alpha}
\newcommand{\na}{\nabla}
\newcommand{\ep}{\epsilon}

\newcommand{\pa}{\partial}
\newcommand{\wh}{\widehat}

\newcommand{\pay}{\partial_{y}}
\newcommand{\CC}{C_{2,\infty}}

\newtheorem{theorem}{Theorem}

\newtheorem{defn}{Definition}
\newtheorem{lem}{Lemma}
\newtheorem{pro}{Proposition}
\newtheorem{rmk}{Remark}

\numberwithin{remark}{section}
\numberwithin{lem}{section}

\newcommand{\norm}[1]{\left\lVert#1\right\rVert}
\newcommand{\abs}[1]{\left\vert#1\right\vert}

\newcommand{\brak}[1]{\langle #1 \rangle}
\newcommand{\grad}{\nabla}

\newcommand\Torus{{\mathbb T}}
\newcommand\Real{{\mathbb R}}

\title{Suppression of blow-up in Parabolic-Parabolic Patlak-Keller-Segel via strictly monotone shear flows}
\date{\today}
\author{He Siming\footnote{\textit{simhe@math.umd.edu}, Department of Mathematics and 
Center for Scientific Computation and Mathematical Modeling, University of Maryland, College Park. }}

\begin{document}
\maketitle
\begin{abstract}
In this paper we consider the parabolic-parabolic Patlak-Keller-Segel models in $\mathbb{T}\times\rr$ with advection by a large strictly monotone shear flow.
Without the shear flow, the model is $L^1$ critical in two dimensions with critical mass $8\pi$: solutions with mass less than $8\pi$ are global in time and there exist solutions with mass larger than $8 \pi$ which blow up in finite time \cite{Schweyer14}.
We show that the additional shear flow, if it is chosen sufficiently large, suppresses one dimension of the dynamics and hence can suppress blow-up. In contrast with the parabolic-elliptic case \cite{BedrossianHe16}, the strong shear flow has destabilizing effect in addition to the enhanced dissipation effect, which make the problem more difficult.

\end{abstract}
\setcounter{tocdepth}{1}
\tableofcontents
\section{Introduction}
In this paper, we consider the two-dimensional parabolic-parabolic Patlak-Keller-Segel equations with additional effect of advection by a shear flow, which model the chemotaxis phenomena in a moving fluid:\begin{subequations}\label{KS advection}
\begin{align}
&\pa_t n+\na \cdot(n \na c)+Au(y)\pa_x n=\de n,\label{KS_1}\\
&\ep\left(\pa_t c+Au(y)\pa_x c\right)= \de c +n-c\label{KS_2},\\
&n(x,y,0)=n_{in}(x,y),\quad c(x,y,0)=c_{in}(x,y),\quad(x,y)\in \mathbb{T} \times\rr.
\end{align}
\end{subequations}
Here $n(x,y,t)$ and $c(x,y,t)$ denote the micro-organism density and the chemo-attractant density, respectively. The divergence free vector field $(Au(y),0)$  represents the underlying fluid velocity. When $Au\equiv 0$, the system is the classical parabolic-parabolic Patlak-Keller-Segel equation modeling chemotaxis in a static environment; see e.g. \cite{Patlak},\cite{KS}. In this case, the first part of \eqref{KS advection} describes the time evolution of the micro-organism density $n$ subject to diffusion and chemo-attractant-triggered aggregation. The second part of \eqref{KS advection} models the time evolution of the chemo-attractant secreted by the micro-organism. The parameter $\ep=0,1$ corresponds to the parabolic-elliptic case and parabolic-parabolic case respectively. When $Au\neq0$, the system \eqref{KS advection} takes into account the advection effect of fluid in the ambient environment.

We focus on the case where $Au\neq 0$ to reflect a scenario of chemotaxis taken place in moving fluid. The question we address is whether one can use a shear flow $Au$ to prevent the micro-organism from undergoing chemotactic blow-up when $u= 0$. It is worth mentioning that system \eqref{KS advection} is one of many attempts to take into account the effect of the moving fluid. For other related models, see \cite{KiselevRyzhik12},\cite{Lorz10},\cite{LiuLorz11},\cite{FrancescoLorzMarkowich10},\cite{DuanLorzMarkowich10},\cite{Lorz12},\cite{ChaeKangLee14},\cite{Winkler12}.
We recall the large literature on the Patlak-Keller-Segel model in the static case ($u=0$),  referring the interested reader to the review \cite{Hortsmann} and the following works \cite{Biler95},\cite{BilerCorriasDolbeault11},\cite{JagerLuckhaus92},\cite{HillenPainter09},\cite{Naito06},\cite{NagaiSenbaYoshida97},\cite{CorriasPerthame06}, \cite{corrias2014existence},
\cite{HillenPainter09},\cite{TaoWinkler12},\cite{CalvezCorriasEbde12},\cite{KozonoSugiyama09}, \cite{CorriasPerthame08},\cite{BlanchetEJDE06},\cite{BlanchetCarrilloMasmoudi08},\cite{BlanchetCalvezCarrillo08},\cite{CalvezCorrias},\cite{BMKS13},\cite{BRB10},\cite{BedrossianKim10},\cite{BedrossianIA10}
.

It is well known that the Patlak-Keller-Segel equation \eqref{KS advection} is $L^1$ critical and the $L^1$ norm of the solution $M:=||n||_1$ is preserved. If there is no underlying moving fluid, i.e., $Au\equiv0$, the existing results for the parabolic-parabolic case ($\ep=1$) can be summarized as follows. In the sub-critical case $M<8\pi$, the global well-posedness of the free energy solution to \eqref{KS advection} is known \cite{CalvezCorrias},\cite{CarrapatosoMischler17}. On the other hand, if $M>8\pi$, it is shown in \cite{Schweyer14} that there exists finite time blow-up solution on $\rr^2$. In higher-dimension, there exist solutions with arbitrary mass which blow up in finite time,\cite{Winkler13}.

In the recent years, progress was made in proving global existence of solution to \eqref{KS advection} in the parabolic-elliptic regime ($\ep=0$) with total mass $M>8\pi$ and $Au\neq 0$. In \cite{KiselevXu15}, it was shown that if the vector field $u$ is relaxation enhancing - a generalization of weakly mixing introduced in \cite{CKRZ08} - with large enough amplitude, the solution $n$ is global in time. The authors proved that due to the mixing property of $u$, the solution undergoes a large growth in its gradient which significantly enhances the dissipation. Once the enhanced dissipation dominates the nonlinear aggregation, suppression of chemotactic blow-up of Patlak-Keller-Segel follows. In \cite{BedrossianHe16}, it is shown that one can use a strong shear flow without degenerate critical points to suppress the blow up in \eqref{KS advection}. The idea in the paper is to exploit the enhanced dissipation effect of shear flow using hypocercivity \cite{BCZ15},\cite{villani2009} and to prove that a large shear flow can in some sense suppress one dimension in parabolic-elliptic PKS system \eqref{KS advection} and hence make 2D $L^1$ subcritical and 3D $L^1$ critical. It is worth mentioning that the enhanced dissipation effect of shear flow is also shown to be important for understanding the stability of the Couette flow in the 2D and 3D Navier-Stokes equations at high Reynolds number,\cite{BMV14},\cite{BGM15I},\cite{BGM15II},\cite{BGM15III}.

In the parabolic-parabolic setting ($\ep=1$), the situation is different. The mixing of the  shear flow has both stabilizing and destabilizing effect on the system \eqref{KS advection}. On the one hand, same as in the parabolic-elliptic case, mixing enhances the dissipation in the micro-organism evolution equation \eqref{KS_1} and hence stabilizes the dynamics. On the other hand, the extra shear flow advection term $Au(y)\pa_x c$ in the chemo-attractant evolution \eqref{KS_2} creates large gradient in the chemical density $c$, which in turn destabilizes the dynamics through the aggregation nonlinearity $\na\cdot(\na c n)$ in the micro-organism evolution \eqref{KS_1}. It is worth noting that this shear flow destabilizing effect does not exist in the parabolic-elliptic regime due to the fast relaxation of chemical density to equilibrium. 
As a result, it is reasonable to expect that an extra smallness assumption is needed to control the mixing destabilizing effect. In this paper, it is assumed that the $x$-dependent part of the initial chemical gradient is small. Since only the $x$-dependent part of $\pay c$ is strongly forced by the shear flow, this smallness restriction is sufficient to control the growth of the chemo-attractant gradient and hence keep the aggregation nonlinearity in  \eqref{KS_1} bounded independent of $A$. Now the situation is similar to the parabolic-elliptic case, hence one can show suppression of chemotactic blow-up through shear flow.


Denote the following projections for function $g(x,y)$:
\begin{align*}
g_0(y)=\frac{1}{2\pi}\int_{-\pi}^\pi g(x,y)dx,\quad g_{\neq}(x,y)=g(x,y)-g_0(y).
\end{align*}
The main theorem of this paper is as follows.
\begin{theorem} \label{thm:2D}
Let $u \in C^3(\rr)$ be a strictly monotone shear flow whose derivative approaches nonzero numbers at $\pm \infty$ and $||u'||_{W^{2,\infty}}<\infty$. Consider the equation 
\eqref{KS advection} subject to initial condition $n_{in} \in H^1 \cap W^{1,\infty}(\Torus\times\rr)$, $c_{in}\in H^2\cap W^{1,\infty}(\Torus\times \rr), \enskip ||\na(c_{in})_{\neq}||_{H^1\cap W^{1,\infty}}\lesssim_u A^{-q},\enskip q>1/2$. Then there exists an $A_0 = A_0(u,\norm{n_{in}}_{H^1\cap W^{1,\infty}},||\na c_{in}||_{H^1\cap W^{1,\infty}})$ such that if $A > A_0$, the solution to \eqref{KS advection} is global in time.
\end{theorem}

We make several remarks concerning the main theorems.
\begin{rmk} In the theorem, we assume that $||\na(c_{in})_{\neq}||_{H^1\cap W^{1,\infty}}\lesssim A^{-q},\enskip q>1/2$ in order to control the shear flow destabilizing effect. A special case of this is $c_{in}\equiv0$, which  corresponds to the situation that at the initial time of the chemotaxis experiment, no chemo-attractant exists in the environment.
\end{rmk}
\begin{rmk} The difficulty is twofold. First we need to construct a hypocercivity functional adapted to the parabolic-parabolic PKS equation, which is significantly more subtle than the one in the parabolic-elliptic case \cite{BedrossianHe16}. Secondly, one needs to control $||\na c_{\neq}||_\infty$ uniformly independent of $A$ for all time. This is delicate due to the shear flow destabilizing effect.
\end{rmk}

\begin{rmk} Since the shear flow is a stationary solution to the Navier-Stokes equation, this result is the first step to proving the suppression of blow-up for the parabolic-parabolic Keller-Segel-Navier-Stokes system.
\end{rmk}


\subsection{Notations}
\subsubsection{Miscellaneous}
Given quantities $X,Y$, if there exists a constant $B$ such that $X \leq BY$, we often write $X\lesssim Y$.
We will moreover use the notation $\lan x\ran:=(1+|x|^2)^{1/2}$.

\subsubsection{Fourier Analysis}
For $f(x,y)$ we define the Fourier transform $\wh{f}(k,y)$ only in terms of variable $x$, and the inverse Fourier transform as follows:
\begin{equation*}
\wh{f}(k,y) = \frac{1}{2\pi}\int_{\Torus} e^{-ikx} f(x,y) dx, \quad \check{g}(x,y) = \sum_{k=-\infty}^\infty g(k,y) e^{ikx}.
\end{equation*}
Define the following orthogonal projections:
\begin{align*}
f_0(t,y) &= \frac{1}{2\pi}\int_{-\pi}^\pi f(t,x,y) dx ,\\
f_{\neq}(t,x,y)& = f(t,x,y) - f_0(t,y).
\end{align*}
Here '$0$' and '$\neq$' stand for ``zero frequency'' and ``non-zero frequencies''.
For any measurable function $m(k)$, we define the Fourier multiplier $m(\pa_{x})f := (m(k)\hat{f}(k,y))^{\vee}$.
\subsubsection{Functional spaces}
The norm for the $L^p$ space is denoted as $||\cdot||_p$ or $||\cdot||_{L^p(\cdot)}$:
\begin{align*}
||f||_p=||f||_{L^p}=\bigg(\int |f|^p dx\bigg)^{1/p},
\end{align*} with natural adjustment when $p$ is $\infty$.
If we need to emphasize the ambient space, we use the second notation, i.e., $||n_{\neq}||_{L^p(\Torus\times\rr)}$. Otherwise, we use the first notation for the sake of simplicity. The Sobolev norm $||\cdot||_{H^s}$ is defined as follow:
\begin{align*}
||f||_{H^s}:=||\lan \grad \ran^s f||_{L^2}.
\end{align*}
For a function of space and time $f = f(t, x)$, we use the following space-time norms:
\begin{align*}
||f||_{L_t^pL_x^q}:=&||||f||_{L_x^q}||_{L_t^p},\\
||f||_{L_t^p H_x^s}:=&||||f||_{H^s_x}||_{L_t^p}.
\end{align*}

The paper is organized as follows: in section 2, we set up the bootstrap argument; in section 3, we prove the enhanced dissipation of the $x$-dependent part of the solution; in section 4, we prove the $L_t^2\dot H_{x,y}^1$ estimate of x-dependent part the micro-organism density; in section 5, we estimate the $x$ independent part of the solution; in section 6, we prove the uniform in time $L^\infty$ estimate of the solution.

\textbf{Acknowledgement} Research was supported by NSF grants DMS16-13911, RNMS11-07444 (KINet), ONR grant N00014-1512094, the Sloan research fellowship FG-2015-66019 
and the Ann G. Wylie Dissertation Fellowship. The author would like to thank his advisor Eitan Tadmor for suggesting this problem and his patient guidance during the years. The author would also like to thank Jacob Bedrossian for many fruitful discussions and careful reading of the draft. The main part of the work was completed when the author was visiting ETH Institute for Theoretical Studies (ETH-ITS) , he would like to thank the institute for the support and hospitality. 

\section{Bootstrap}
As in the paper \cite{BedrossianHe16}, since the enhanced dissipation does not act on the nullspace of the advection term, it is reasonable to rescale in time and decompose the solution as follows \begin{subequations}
\begin{align}\label{zeromode}
\pa_t n_0+&\frac{1}{A}\pa_y(\pa_y c_0 n_0)+\frac{1}{A}(\na\cdot(\na c_{\neq} n_{\neq}))_0=\frac{1}{A}\pa_{yy} n_0,\\
\pa_t c_0=&\frac{1}{A}\de c_0+\frac{1}{A}n_0-\frac{1}{A}c_0;
\end{align}
\end{subequations}
and,
\begin{subequations}
\begin{align}
\pa_t n_{\neq} +&u(y)\pa_x n_{\neq} +\frac{1}{A}\na\cdot(\na c_{\neq} n_0)+\frac{1}{A}\pay(\pay c_0n_{\neq})
+\frac{1}{A}(\na\cdot(\na c_{\neq} n_{\neq}))_{\neq}= \frac{1}{A}\de n_{\neq},\label{nonzeromodeN}\\
\pa_t c_{\neq}+&u(y)\pa_x c_{\neq}=\frac{1}{A}\de c_{\neq}+\frac{1}{A}n_{\neq}-\frac{1}{A}c_{\neq}\label{nonzeromodeC}.
\end{align}
\end{subequations}
As in the paper \cite{BCZ15}, it is convenient to consider equations \eqref{nonzeromodeN} and \eqref{nonzeromodeC} after applying the Fourier transform \emph{only in $x$}.
Applying to both sides of (\ref{nonzeromodeN},\ref{nonzeromodeC}) we have,
\begin{subequations}
\begin{align}\label{nonzeromodeF}
\pa_t \widehat{n}_{k}+&NL_k+L_k+u(y)ik \widehat{n}_{k}=\frac{1}{A}(\pa_{yy}-|k|^2) \widehat n_{k},\\
\pa_t \wh{c}_k+&u(y)ik\wh c_k=\frac{1}{A}(\pa_{yy}-|k|^2) \wh{c}_{k}+\frac{1}{A}\wh{n}_{k}-\frac{1}{A}\wh{c}_{k},\quad k\neq 0,
\end{align}
\end{subequations}
where $L_k,NL_k$ are defined as follows:
\begin{align}
NL_k:=&\frac{1}{A}\sum_{\ell \neq 0,k} \pa_y(\pa_y \widehat{c}_{k-\ell} \widehat{n}_\ell) - \frac{1}{A}k\sum_{\ell \neq 0,k}(k-\ell) \widehat{c}_{k-\ell} \widehat{n}_\ell, \label{NL} \\
L_k:=&\frac{1}{A}\pa_y(\pa_y c_0 \widehat{n}_{k})+\frac{1}{A}\na\cdot(\na \widehat{c}_{k} n_0)=\frac{1}{A}\pa_y(\pa_y c_0 \widehat{n}_{k})-\frac{1}{A}k^2\wh{c}_kn_0+\frac{1}{A}\pa_y(\pa_y \widehat{c}_{k} n_0). \label{L}
\end{align}
Here, the $L$ refers to ``linear with respect to the nonzero frequencies'' and $NL$ refers to ``nonlinear with respect to the nonzero frequencies''.

As is standard in the study of nonlinear mixing, we use a bootstrap argument to prove the main theorem. For constants $C_{ED},C_{n_0,L^2}$, $C_{n_0,\dot{H}^1}$, $C_{n,\infty}$, $C_{\na c_{\neq},\infty}$ and $A_0$ determined by the proof, define $T_\star$ to be the end-point of the largest interval $[0,T_\star]$ such that the following hypotheses hold for all $t \leq T_\star$:

\begin{subequations}
(1) Nonzero mode $L_t^2 \dot{H}_{x,y}^1$ estimate:
\bel\ba\label{H1}
\frac{1}{A}\int_0^{T_\star}||\nabla_{x,y} n_{\neq}||_2^2dt\leq& 8||n_{in}||_2^2;\\
\ea\eel

(2) Nonzero mode enhanced dissipation estimate:
\bel\label{H2}\ba
||n_{\neq}(t)||_2^2+||\na c_{\neq}(t)||_2^2\leq& 4C_{ED}(||n_{in}||_{H^1}^2+1)e^{-\frac{\eta t}{A^{1/3}}}, \quad\forall t < T_\star,
\ea\eel
where $\eta$ is a small constant depending only on $u$.

(3) Uniform in time estimates on the zero mode:
\bel\ba\label{H3}
||\pa_y c_0||_{L^\infty(0,T_\star; L^\infty_{y})}+||n_0||_{L^\infty_t(0,T_\star;L^2_{y})}\leq& 4C_{n_0,L^2},\\
||\pa_y n_0||_{L^\infty_t(0,T_\star;L^2_{y})}\leq& 4C_{n_0,\dot{H}^1};\\
\ea\eel

(4) $L^\infty$ estimate of the solution $n$:
\bel\label{H4}\ba
||n||_{L^\infty_t(0,T_\star;L^\infty_{x,y})}\leq& 4C_{n,\infty};
\ea\eel

(5) $L^\infty$ estimate of the $x$-dependent part of the chemical gradient $\na c_{\neq}$:
\bel\label{H5}\ba
||\na c_{\neq}||_{L^\infty_t(0,T_\star;L^\infty_{x,y})}\leq& 4C_{\na c_{\neq},\infty}.
\ea\eel

Furthermore, we define the following constant to simplify the notation:
\begin{equation}
\CC:=1+M+C_{ED}^{1/2}||n_{in}||_{H^1}+C_{n_0,L^2}+C_{n,\infty}+C_{\na c_{\neq},\infty}+||\na( c_{in})_0||_{H^1\cap W^{1,\infty}}.
\end{equation}
\end{subequations}
Note that $C_{n_0, \dot H^1}$ is not included in $\CC$.

The goal is to prove the following improvement to the above hypotheses:
\begin{pro}\label{prop:boot2D}
For all $n_{in}$, $c_{in}$ and $u$ satisfying the assumption of Theorem \ref{thm:2D}, there exists an $A_0=A_0(u,\norm{n_{in}}_{H^1\cap L^\infty},\norm{\na c_{in}}_{H^1\cap W^{1,\infty}})$ such that if $A > A_0$ then the following conclusions hold on the interval $[0,T_\star]$:

\noindent
\begin{subequations}
(1) \bel\ba\label{ctrl:L2H1}
\frac{1}{A}\int_0^{T_\star}||\nabla_{x,y} {n}_{\neq}||_2^2dt\leq& 4||n_{in}||_2^2;\\
\ea\eel
(2)
 \bel\label{ctrl:ED}\ba
||{n}_{\neq}(t)||_2^2+||\na c_{\neq}(t)||_2^2\leq& 2C_{ED}(||n_{in}||_{H^1}^2+1)e^{-\frac{\eta t}{A^{1/3}}}, \quad\forall t < T_\star;
\ea\eel
(3) \bel\ba\label{ctrl:ZeroMd}
||\pay c_0||_{L^\infty_t(0,T_\star;L^\infty_y)}+||{n}_0||_{L^\infty_t(0,T_\star; L^2_{y})}\leq& 2C_{n_0,L^2},\\
||\pa_y {n}_0||_{L^\infty_t(0,T_\star;L^2_{y})}\leq& 2C_{n_0,\dot{H}^1};\\
\ea\eel
(4) \bel\label{ctrl:n_Linf}\ba
||n||_{L^\infty(0,T_\star;L^\infty_{x,y})}\leq& 2C_{n,\infty};
\ea\eel
(5) \bel\label{ctrl:na_c_Linf}\ba
||\na c_{\neq}||_{L^\infty_t(0,T_\star;L^\infty_{x,y})}\leq& 2C_{\na c_{\neq},\infty}.
\ea\eel
\end{subequations}
\end{pro}
\begin{rmk}
The proposition together with the local wellposedness of the equation \eqref{KS advection} implies the Theorem \ref{thm:2D}.
\end{rmk}
\begin{rmk} The constants in the proof are determined in the following order
\begin{equation}
C_{ED}\Rightarrow C_{n_0,L^2}\Rightarrow C_{n,\infty}, C_{\na c_{\neq},\infty}\Rightarrow C_{n_0,\dot H^1}\Rightarrow A_0.
\end{equation}
The magnitude of the flow $A_0$ will be chosen large depending on the constants in the hypotheses and the intermediate constants in the proof.

\end{rmk}

\begin{rmk}
We need to control the shear flow destabilizing effect in the proof of \eqref{ctrl:ED}, \eqref{ctrl:n_Linf} and \eqref{ctrl:na_c_Linf}.
\end{rmk}
\section{Enhanced dissipation estimate \eqref{ctrl:ED}}
\subsection{Enhanced dissipation functional $\mathcal{F}$}
In this subsection, we construct the functional $\mathcal{F}$ to exploit the enhanced dissipation in the equation \eqref{KS advection}. Similar to the paper \cite{BedrossianHe16}, the following hypocoercivity functional (\cite{BCZ15},\cite{villani2009}) serves as a building block for the construction:
\bel\label{Phi k}
\Phi_k[f(t)] =||\widehat{f}_k(t)||_2^2+||\sqrt{\al}\pa_y \widehat{f}_k(t)||_2^2+2kRe\lan i\beta u' \widehat{f}_k(t),\pa_y\widehat{f}_k
(t)\ran+|k|^2||\sqrt{\gamma}u' \widehat{f}_k(t)||_2^2;
\eel
\bel\label{Phi}
\Phi[f(t)]=\sum_{k\neq 0} \Phi_k[f(t)] =||f_{\neq}(t)||_{2}^2+||\sqrt{\al}\pa_y f_{\neq}(t)||_2^2 + 2\lan \beta u' \pa_xf_{\neq}(t),\pa_yf_{\neq}
(t)\ran+||\sqrt{\gamma}u'|\pa_x| f_{\neq}(t)||_2^2.
\eel
Here $\alpha,\beta$, and $\gamma$ are $k$-dependent constants (and hence should be interpreted as Fourier multipliers) satisfying
\begin{subequations} \label{def:abg}
\begin{align}
\alpha(A,k) & = \epsilon_\alpha A^{-2/3}\abs{k}^{-2/3} \\
\beta(A,k) & = \epsilon_\beta A^{-1/3}\abs{k}^{-4/3} \\
\gamma(A,k) & = \epsilon_\gamma \abs{k}^{-2},
\end{align}
\end{subequations}
where $\epsilon_\alpha$, $\epsilon_\beta$, and $\epsilon_{\gamma}$ are small constants depending only on $u$ chosen in \cite{BCZ15}. Among other things, these are chosen such that 8$\beta^2 \leq \alpha \gamma$. Since we are concerned with strictly monotone shear flows instead of nondegenerate shear flows, we employ slightly different multipliers $\al,\beta,\gamma$ from  the ones in the paper \cite{BedrossianHe16}. Notice that in \cite{BCZ15} for treating general situations one must also take $\alpha,\beta$, and $\gamma$ to be $y$-dependent, however, as suggested by \cite{BeckWayne11}, this is not necessary to treat strictly monotone shear flows  with $y \in \Real$.
The parameters $\epsilon_\alpha$, $\epsilon_\beta$, and $\epsilon_{\gamma}$ are tuned such that,
\begin{align}
\Phi_k[f] \approx \norm{\widehat{f}_k}_2^2+\norm{\sqrt{\al}\pa_y \widehat{f}_k}_2^2+|k|^2\norm{\sqrt{\gamma}u'\widehat{f}_k}_2^2,
\end{align}
and hence
\begin{align}\Phi_k[f] \approx \norm{\widehat{f}_k}_{2}^2 +  \abs{k}^{-2/3} A^{-2/3} \norm{\partial_y \wh{f}_k}_2^2. \label{ineq:PhiEquiv}
\end{align}
As a result, $\Phi_k[f(t)]$ is equivalent to the $H^1$ norm of $f_{k}$ but with constants that depend on $A$ and $k$.
The primary step in the results of \cite{BCZ15} is that for $u(y)$ satisfying the hypotheses in Theorem \ref{thm:2D}, then for the passive scalar equation on $\mathbb{T}\times \Real$,
\bel\label{PS}
\pa_t f+u(y)\pa_x f= \frac{1}{A}\de f,
\eel
the norm $\Phi_k[f(t)]$ satisfies the following differential inequality for some small constant $\tilde{\ep}$  independent of $k, A$ (but depending on $u$):
\be
\frac{d}{dt}\Phi_k[f(t)]\leq -\tilde{\ep}\frac{|k|^{2/3}}{A^{1/3}}\Phi_k[f(t)].
\ee
Note that the decay rate of the functional $\Phi_k[f]$ $\big(=\frac{\tilde{\ep}}{A^{1/3}}\big)$ is much larger than the classical heat decay rate $\big(=\frac{1}{A}\big)$ for the passive scalar equation \eqref{PS} when $A$ is chosen big. This is the enhanced dissipation effect of the shear flow.

Recall the estimate of the time evolution of $\Phi_k[f(t)]$ in \cite{BCZ15}.
\begin{pro}(\cite{BCZ15})\label{thm d/dt Phi k f}
Consider the solution to the passive scalar equation \eqref{PS}. For $\tilde{\epsilon}$ sufficiently small depending only on $u$, there holds,
\begin{align}
\frac{d}{dt}\Phi_k[f(t)] \leq&-\frac{\tilde{\ep}}{2} \frac{|k|^{2/3}}{A^{1/3}}||\widehat{f}_k||_2^2-\frac{\tilde{\ep}}{2}\frac{|k|^{2/3}}{A^{1/3}}||\sqrt{\al}\pa_y \widehat{f}_k||_2^2-\frac{\tilde{\ep}}{2}\frac{|k|^{8/3}}{A^{1/3}}||\sqrt{\gamma}u'\widehat{f}_k||_2^2-\frac{1}{4A}||\pa_y\wh{f}_k||_2^2\nonumber\\
&-\frac{1}{2}|k|^2||\sqrt{\beta}u' \wh{f}_k||_2^2-\frac{1}{2A}|k|^2||\wh{f}_k||_2^2-\frac{1}{4A}||\sqrt{\al}\pa_{yy}\wh{f}_k||_2^2\nonumber\\
&-\frac{1}{4A}|k|^4||\sqrt{\gamma}u' \wh{f}_k||_2^2-\frac{1}{4A}|k|^2||\sqrt{\gamma}u'\pa_y\wh{f}_k||_2^2\nonumber\\
=:&\mathcal{N}_k[f].\label{ddtPhik_f}
\end{align}
\end{pro}
\begin{rmk}
The notation "$\mathcal{N}$" stands for "negative terms".
\end{rmk}
The functional we construct to exploit the enhanced dissipation effect in the equation \eqref{KS advection} is the following:
\begin{defn} Define the functional $\mathcal{F}$ as
\begin{align}\label{ParabolicParabolicHypoFunctional}
\mathcal{F}_k:=&\Phi_k[n_{\neq}]+\Phi_k[\pa_y c_{\neq}]+\Phi_k[\pa_x c_{\neq}]+A|k|\Phi_k[c_{\neq}];\\
\mathcal{F}:=&\sum_{k\neq0}\Phi_k[n_{\neq}]+\sum_{k\neq0}\Phi_k[\pa_y c_{\neq}]+\sum_{k\neq0}\Phi_k[\pa_x c_{\neq}]+\sum_{k\neq0}A|k|\Phi_k[c_{\neq}]=\sum_{k\neq0}\mathcal{F}_k.
\end{align}
\end{defn}

The goal in this subsection is to show that
\begin{theorem} Assume the hypothesis of Proposition \ref{prop:boot2D}. There exists a constant $\eta>0$ depending only on $u$ such that the following time decay estimate holds if $A$ is chosen large enough
\begin{align}\label{TimeDerivativeOfF}
\frac{d}{dt}\mathcal{F}
\leq -\frac{\eta}{A^{1/3}}\mathcal{F}.
\end{align}
\end{theorem}

\begin{rmk} Thanks to the assumption on the initial chemical gradient
\bel\label{IntialConditionForC}
||\na(c_{in})_{\neq}||_{H^1}\lesssim A^{-q},\quad q>1/2,
\eel
the initial value $F(0)$ is bounded
\begin{equation}
F(0)\leq C(\ep_\al,\ep_\beta,\ep_\gamma, u)\bigg(||(n_{in})_{\neq}||_{H^1}^2+||(\na c_{in})_{\neq}||_{H^1}^2(1+A)\bigg)\leq C_{ED}(||(n_{in})_{\neq}||_{H^1}^2+1),
\end{equation}
Combining this with the fact that $\mathcal{F}(t)\geq ||n_{\neq}(t)||_2^2+||\na c_{\neq}(t)||_2^2$, we conclude that the estimate \eqref{TimeDerivativeOfF} implies \eqref{ctrl:ED}.
\end{rmk}

In order to show the idea behind the construction of the functional $\mathcal{F}$, we first list all the related equations here:
\begin{align}
\pa_t \wh{n}_{k}=&\frac{1}{A}\pa_{yy}\wh{n}_k-\frac{|k|^2}{A}\wh{n}_k-u(y)ik\wh{n}_k-L_k-NL_k;\label{nkeq}\\
\pa_t \pa_y \wh{c}_k=&\frac{1}{A}\pa_{yy}\pa_y\wh{c}_k-\frac{|k|^2}{A}\pa_y\wh{c}_k-u(y)ik\pa_y\wh{c}_k-u'(y)ik\wh{c}_k+\frac{1}{A}\pa_y \wh{n}_k-\frac{1}{A}\pay\wh{c}_k;\label{payckeq}\\
\pa_t \wh{c}_k=&\frac{1}{A}\pa_{yy}\wh{c}_k-\frac{|k|^2}{A}\wh{c}_k-u(y)ik\wh{c}_k+\frac{1}{A}\wh{n}_k-\frac{1}{A}\wh{c}_k,\label{ckeq}
\end{align}
where $L_k,NL_k$ are defined as in \eqref{L} and \eqref{NL}. 
Our primary goal is to obtain the $L^2$ enhanced dissipation estimate of $n_{\neq}$. However,   we are not able to close the estimate on ${d}\Phi_k[n_{\neq}]/{dt}$ without further information about the chemical gradient $\pay c_{\neq}$. Specifically speaking, the terms in $L_k,\enskip NL_k$ involving $\pa_y(\pa_y\wh{c}_{\neq} \wh{n}_{0,\neq})$ cannot be absorbed by the negative terms in $d\Phi_k[n_{\neq}]/dt$. Therefore, in the first step, we add $\Phi_k[\na c_{\neq}]$ in the functional $\mathcal{F}$ to make use of the extra negative terms in ${d}\Phi[\na c_{\neq}]/dt$. The drawback is that it introduces shear flow destabilizing effect into the functional since problematic terms involving $-u'(y)ik\wh{c}_k$ are created. These terms will typically involve large powers of $A$ and $|k|$. In the second step, we add the term $A|k|\Phi_k[c_{\neq}]$ in $\mathcal{F}$ to compensate for this shear flow destabilizing effect
. Finally, we show that the negative terms in $d\Phi_k[n_{\neq}]/{dt}$ absorb all terms involving $n_{\neq}$ in $A|k|\Phi_k[c_{\neq}]$. By completing this loop, we have shown that all the terms are absorbed by the negative terms in the time derivative of $\mathcal{F}$ and the exponential decay \eqref{TimeDerivativeOfF} follows.

\begin{pro}\label{thm d/dt Phi n payc c k}
For $\tilde{\epsilon}$ sufficiently small depending only on $u$, there holds,
\begin{align}
\frac{d}{dt}\Phi_k[n_{\neq}(t)]\leq 
&\mathcal{N}_k[n_{\neq}]+\bigg\{2Re\lan -L_k ,\widehat{n}_k\ran
-2Re\lan \al \pa_{yy}\widehat{n}_k,-L_k\ran-2kRe[\lan i\beta u'L_k,\pa_y\widehat{n}_k\ran+\lan i\beta u'\widehat{n}_k,\pa_y L_k\ran]\nonumber\\
&+2|k|^2Re\lan \gamma (u')^2\widehat{n}_k,-L_k\ran\bigg\}\nonumber\\
&+\bigg\{-2Re\lan NL_k,\widehat{n}_k\ran
+2Re\lan  \al \pa_{yy}\widehat{n}_k,NL_k\ran-2kRe[\lan i\beta u'NL_k,\pa_y\widehat{n}_k\ran+\lan i\beta u'\widehat{n}_k,\pa_y NL_k\ran]\nonumber\\
&-2|k|^2Re\lan \gamma (u')^2\widehat{n}_k,NL_k\ran\bigg\}\nonumber\\
=:&\mathcal{N}_{n,k}+\{L_k^1+L_k^{\al}+L_k^{\beta}+L_k^{\gamma}\}+\{NL_k^1+NL_k^{\al}+NL_k^{\beta}+NL_k^{\gamma}\}. \label{ddtPhikN}
\end{align}
Recall that $\mathcal{N}_k$ is defined in \eqref{ddtPhik_f} and $L_k,NL_k$ are defined in (\ref{NL},\ref{L}). The time derivative of $\Phi_k[\pa_y c_{\neq}],\Phi_k[\pa_x c_{\neq}]$ are bounded,
\begin{align}
\frac{d}{dt}\Phi_k[\pa_y c_{\neq}(t)] \leq
&\mathcal{N}_k[\pa_y c_{\neq}]+\bigg\{2Re\left\lan \frac{\pa_y \wh{n}_k}{A}  ,\pa_y\widehat{c}_k\right\ran
-2Re\left\lan \al \pa_{yy}\pa_y\widehat{c}_k,\frac{\pa_y \wh{n}_k}{A} \right\ran+2kRe\bigg[\left\lan i\beta u'\frac{\pa_y \wh{n}_k}{A} ,\pa_y\pa_y\widehat{c}_k\right\ran\nonumber\\
&+\left\lan i\beta u'\pa_y\widehat{c}_k,\frac{\pa_y \pa_y \wh{n}_k}{A} \right\ran\bigg]+2|k|^2Re\left\lan \gamma (u')^2\pa_y\widehat{c}_k,\frac{\pa_y \wh{n}_k}{A} \right\ran\bigg\}\nonumber\\
&+\bigg\{-2Re\left\lan u'ik\wh{c}_k,\pa_y\widehat{c}_k\right\ran
+2Re\left\lan  \al \pa_{yy}\pa_y\widehat{c}_k, u'ik\wh{c}_k\right\ran-2kRe\bigg[\left\lan i\beta (u')^2ik\wh{c}_k,\pa_y\pa_y\widehat{c}_k\right\ran\nonumber\\
&+\lan i\beta u'\pa_y\widehat{c}_k,\pa_y ( u'ik\wh{c}_k)\ran\bigg]\nonumber-2|k|^2Re\lan \gamma (u')^2\pa_y\widehat{c}_k, u'ik\wh{c}_k\ran\bigg\}-4k Re\left\lan i\beta u' \frac{\pay \wh{c}_k}{A},\pa_{yy}\wh{c}_k\right\ran\nonumber\\
=:&\mathcal{N}_{\pa_y c,k}+\{T_{\pa_y c,1;k}^1+T_{\pa_y c, 1;k}^{\al}+T_{\pa_y c,1;k}^{\beta}+T_{\pa_y c,1 ;k}^{\gamma}\}+\{T_{\pa_y c,2;k}^1+T_{\pa_y c, 2;k}^{\al}+T_{\pa_y c,2;k}^{\beta}+T_{\pa_y c, 2;k}^{\gamma}\}+T_{\pay c, 3;k}^\beta, \label{ddtPhikPayC}
\end{align}

\begin{align}
\frac{d}{dt}\Phi_k[\pa_x c_{\neq}(t)] \leq
&\mathcal{N}_k[\pa_x c_{\neq}]+\bigg\{2Re\left\lan \frac{ik \wh{n}_k}{A}  ,ik\widehat{c}_k\right\ran
-2Re\left\lan \al \pa_{yy}ik\widehat{c}_k,\frac{ik \wh{n}_k}{A}\right \ran+2kRe\bigg[\left\lan i\beta u'\frac{ik \wh{n}_k}{A} ,\pa_y ik\widehat{c}_k\right\ran\nonumber\\
&+\left\lan i\beta u'ik\widehat{c}_k,\frac{ik \pa_y \wh{n}_k}{A} \right\ran\bigg]+2|k|^2Re\left\lan \gamma (u')^2ik\widehat{c}_k,\frac{ik\wh{n}_k}{A}\right\ran\bigg\}-4kRe\left\lan i\beta u'\frac{ik\wh{c}_k}{A},ik\pay \wh{c}_k\right\ran\nonumber\\
=:&\mathcal{N}_{\pa_x c,k}+\{T_{\pa_x c,1;k}^1+T_{\pa_x c, 1;k}^{\al}+T_{\pa_x c,1;k}^{\beta}+T_{\pa_x c,1 ;k}^{\gamma}\}+T_{\pa_x c,2;k}^\beta. \label{ddtPhikPaxC}
\end{align}
The time derivative of $A|k|\Phi_k[c_{\neq}]$ is bounded,
\begin{align}
\frac{d}{dt} A|k|\Phi_k[c_{\neq}(t)] \leq& A|k|\mathcal{N}_k[c_{\neq}]
+ A|k|\bigg\{2Re\left\lan \frac{\wh{n}_k}{A} ,\widehat{c}_k\right\ran
-2Re\left\lan \al \pa_{yy}\widehat{c}_k,\frac{\wh{n}_k}{A}\right\ran+2kRe\bigg[\left\lan i\beta u'\frac{\wh{n}_k}{A},\pa_y\widehat{c}_k\right\ran\nonumber\\
&+\left\lan i\beta u'\widehat{c}_k,\frac{\pa_y \wh{n}_k}{A}\right\ran\bigg]+2|k|^2Re\left\lan \gamma (u')^2\widehat{c}_k,\frac{\wh{n}_k}{A}\right\ran\bigg\}-A|k|4kRe\left\lan i\beta u'\frac{\wh{c}_k}{A},\pay \wh{c}_k\right\ran\nonumber\\
=:& A|k|\mathcal{N}_{c,k}+ A|k|\{T_{c,1;k}^1+T_{c,1;k}^{\al}+T_{c,1;k}^{\beta}+T_{c,1;k}^{\gamma}\}+A|k|T_{c,2;k}^\beta. \label{ddtPhikC}
\end{align}
\end{pro}
\begin{proof}Applying the equations \eqref{ddtPhik_f}, \eqref{nkeq}, \eqref{payckeq}, \eqref{ckeq} and integration by parts, the estimates follow.
\end{proof}

The remaining part of this section is organized as follows: in section 3.2, we estimate all the terms in \eqref{ddtPhikC}; in section 3.3, we estimate \eqref{ddtPhikPayC} and \eqref{ddtPhikPaxC}; in section 3.4, we estimate \eqref{ddtPhikN}.

\subsection{Time evolution estimates: $A|k|\frac{d}{dt}\Phi_k[c_{\neq}]$}
In this subsection, we estimate terms in \eqref{ddtPhikC}. First the $A |k| T_{c,1;k}^1$ term in \eqref{ddtPhikC} can be estimated using H\"older inequality and Young's inequality:
\begin{align}\label{Tc1k1}
T_{c,1;k}^1=2Re\bigg\lan \frac{\wh{n}_k}{A},\wh{c}_k\bigg\ran\leq \frac{B/\tilde{\ep}}{A^{5/3}|k|^{2/3}}||\wh{n}_k||_2^2+\frac{|k|^{2/3}}{A^{1/3}B}\tilde{\ep}||\wh{c}_k||_2^2.
\end{align}
We show that $ A|k| T_{c,1;k}^1$ is consistent with \eqref{TimeDerivativeOfF} given that $B$, then $A$, are chosen large. For the second term in \eqref{Tc1k1}, it can be absorbed by the negative term $A|k|\mathcal{N}_k[c_{\neq}]$ in \eqref{ddtPhikC} given $B$ chosen large enough. For the first term, we can use the negative term $- \frac{\tilde{\ep}}{2}\frac{|k|^{2/3}}{A^{1/3}}||\wh{n}_k||_2^2$  in \eqref{ddtPhikN} to absorb it given $A$ chosen large enough compared to $B$ and $1/\tilde{\ep}$, i.e.
\begin{align*}
A|k|\frac{B/\tilde{\ep}}{A^{5/3}|k|^{2/3}}||\wh{n}_k||_2^2-\frac{1}{2}\tilde{\ep}\frac{|k|^{2/3}}{A^{1/3}}||\wh{n}_k||_2^2\leq-\frac{7}{16}\tilde{\ep}\frac{|k|^{2/3}}{A^{1/3}}||\wh{n}_k||_2^2.
\end{align*} 
The second term $ A |k| T_{c,1;k}^\al$ in \eqref{ddtPhikC} is estimated using H\"older inequality, Young's inequality and the definition of $\al$ \eqref{def:abg}:
\begin{align*}
T_{c,1;k}^\al
\leq \frac{1}{AB}||\sqrt{\al}\pa_{yy}\wh{c}_k||_2^2+\frac{B}{A^{5/3}|k|^{2/3}}||\wh{n}_k||_2^2,
\end{align*}
which by \eqref{ddtPhik_f}, \eqref{ddtPhikN} and \eqref{ddtPhikC} is consistent with \eqref{TimeDerivativeOfF} given $A$ large. 
For the $A|k|T_{c,1;k}^\beta$ term in \eqref{ddtPhikC}, we estimate it using the fact that $||u''||_\infty\leq C$, the definition of $\beta$ \eqref{def:abg}, H\"older inequality and Young's inequality as follows
\begin{align*}
T_{c,1;k}^\beta=&4kRe\bigg\lan i\beta u'\frac{\wh{n}_k}{A},\pa_y\wh{c}_k\bigg\ran-2kRe\bigg\lan i\beta u''\wh{c}_k,\frac{\wh n_k}{A}\bigg\ran\\
\lesssim&\frac{B|k|^{2/3}}{A^{4/3}}||\sqrt{\beta}u'\wh{n}_k||_2^2+\frac{||\pa_y \wh{c}_k||_2^2}{AB}+\frac{||\wh{c}_k||_2^2}{BA^{1/3}}+\frac{B||\wh{n}_k||_2^2}{A^{7/3}|k|^{2/3}},
\end{align*}
which by \eqref{ddtPhik_f}, \eqref{ddtPhikN} and \eqref{ddtPhikC} is consistent with \eqref{TimeDerivativeOfF} given $B$, then $A$ large. 
Similarly, the $A|k|T_{c,1;k}^\gamma$ term in \eqref{ddtPhikC} can be estimated using H\"older inequality and Young's inequality
\begin{align*}
T_{c,1;k}^\gamma=2|k|^2Re\bigg\lan \gamma (u')^2 \wh{c}_k,\frac{\wh{n}_k}{A}\bigg\ran\leq \frac{|k|^{8/3}}{A^{1/3}B}||\sqrt{\gamma}u'\wh{c}_k||_2^2+\frac{B|k|^{4/3}}{A^{5/3}}||\sqrt{\gamma}u'\wh{n}_k||_2^2,
\end{align*}
which is consistent with \eqref{TimeDerivativeOfF} given that $B$, then $A$, are chosen large enough thanks to \eqref{ddtPhik_f}, \eqref{ddtPhikN} and \eqref{ddtPhikC}. 
The $A|k|T_{c,2;k}^\beta$ term in \eqref{ddtPhikC} can be estimated using H\"older inequality,  Young's inequality and the definition of $\beta$ \eqref{def:abg} as follows
\begin{align}\label{Tc2kbeta}
A|k|T_{c,2;k}^\beta\leq A|k|\left(\frac{8|k|^2||\sqrt{\beta}u'\wh{c}_k||_2^2}{A}+\frac{||\pa_y \wh{c}_k||_2^2}{2A^{4/3}}\right),
\end{align}
which can be absorbed by $A|k|\mathcal{N}_k[c_{\neq}]$ in \eqref{ddtPhikC} given that $A$ is chosen large enough. 
This completes the estimation of all the terms in \eqref{ddtPhikC}.

\subsection{Time evolution estimates: $\frac{d}{dt}\Phi[\na c_{\neq}]$}
In this subsection, we estimate the time evolution of $\Phi[\na c_{\neq}]$  \eqref{ddtPhikPayC} and \eqref{ddtPhikPaxC}. We start by estimating the terms in $\frac{d}{dt}\Phi[\pa_y c_{\neq}]$ since they involve strong shear flow destabilizing effect. 
First we estimate the term $T_{\pa_y c, 2;k}^1$ in \eqref{ddtPhikPayC} using the definition of $\beta$ \eqref{def:abg}, H\"older inequality and Young's inequality as follows:
\begin{align*}T_{\pa_y c, 2;k}^1
\lesssim \frac{|k|^{2/3}||\pa_y \wh{c}_k||_2^2}{BA^{1/3}}+B \left(A^{2/3}|k|^{2/3}\right)|k|^2||\sqrt{\beta}u' \wh{c}_k||_2^2.
\end{align*}
Now we see that the first term is absorbed by the negative terms in \eqref{ddtPhikPayC} given $B$ chosen large enough, and the second term can be absorbed by the term $-A|k|^3||\sqrt{\beta}u'\wh{c}_k||_2^2$ in \eqref{ddtPhikC} given $A$ chosen large enough. Now we see that this term is consistent with \eqref{TimeDerivativeOfF}.
Next, combining the definition of $\al,\beta$ \eqref{def:abg}, H\"older inequality and Young's inequality, the $\al$ term in $T_{\pay c,2;k}^\al$ can be estimated as follows:
\begin{align*}T_{\pay c,2;k}^\al
\lesssim \frac{1}{AB}||\sqrt{\al}\pa_{yyy}\wh{c}_k||_2^2+B\left(A^{2/3}|k|^{2/3}\right)|k|^2||\sqrt{\beta}u'\wh{c}_k||_2^2,
\end{align*}
which is consistent with \eqref{TimeDerivativeOfF} given that $B$, then $A$, are chosen large.
For the first $\beta$ term in $T_{\na c,2 ;k}^\beta$, combining the definition of $\beta$ \eqref{def:abg}, the fact that $||u'||_{W^{1,\infty}}\leq C$, integration by parts, H\"older inequality and Young's inequality yields
\begin{align*}
2kRe\lan i\beta u'(-u'ik\wh{c}_k),\pa_{yy}\wh{c}_k\ran=&-2|k|^2 Re\lan \beta 2u'u'' \wh{c}_k,\pa_y\wh{c}_k\ran-2|k|^2Re\lan \beta u'^2\pa_y \wh{c}_k,\pa_y \wh{c}_k\ran\\
\leq&2|k|^2||\sqrt{\beta}u''\wh{c}_k||_2^2+2|k|^2||\sqrt{\beta}u'\pa_y \wh{c}_k||_2^2-2|k|^2||\sqrt{\beta}u'\pa_y\wh{c}_k||_2^2\\
\lesssim&\frac{|k|^{2/3}}{A^{1/3}}||\wh{c}_k||_2^2,
\end{align*}
which can be absorbed by the negative term $A|k|\mathcal{N}_k[c_{\neq}]$ in \eqref{ddtPhikC} given $A$ large enough. By applying integration by parts, we see that the second $\beta$ term in $T_{\pay c,2;k}^\beta$ is equivalent to the first one up to the following term, which can be estimated using the definition of $\beta$ \eqref{def:abg}, $||u''||_\infty\leq C$, H\"older inequality and Young's inequality
\begin{align*}
2kRe\lan i\beta u''\pa_y \wh{c}_k,u'ik\wh{c}_k\ran\lesssim\frac{|k|^{2/3}||\pa_y\wh{c}_k||_2^2}{A^{1/3}B}+B|k|^2||\sqrt{\beta}u'\wh{c}_k||_2^2.
\end{align*}
Since the first terms can be absorbed by $\mathcal{N}_k[\pay c_{\neq}]$ and the second term  can be absorbed by $A|k|\mathcal{N}_k[c_{\neq}]$, this is consistent with \eqref{TimeDerivativeOfF} given that $B$, then $A$, are chosen large.
The $T_{\pay c,2;k}^\gamma$ term in \eqref{ddtPhikPayC} can be estimated using $||u'||_\infty\leq C$, H\"older inequality and Young's inequality as follows:
\begin{align*}
T_{\pay c,2;k}^\gamma
\lesssim \frac{|k|^{8/3}}{A^{1/3}B}||\sqrt{\gamma}u'\pa_y\wh{c}_k||_2^2+B\left(A^{2/3}|k|^{2/3}\right)\frac{||\sqrt{\gamma}u'\wh{c}_k||_2^2|k|^{8/3}}{A^{1/3}}.
\end{align*}
Now we see that the first term is absorbed by the negative term $\mathcal{N}_k[\pay c_{\neq}]$ in \eqref{ddtPhikPayC} if $B$ is chosen large, and the second term is absorbed by $A|k|\mathcal{N}_k[c_{\neq}]
$ in \eqref{ddtPhikC} given that $A$ is chosen large. This finishes the estimation of the terms $T_{\pay c,2;k}^{(\cdot)}$ in \eqref{ddtPhikPayC}.

For the terms of the form $T_{\pay c,1;k}^{(\cdot)}$ in \eqref{ddtPhikPayC}, we will use the negative terms in \eqref{ddtPhikN} and \eqref{ddtPhikPayC} to absorb them. For the $T_{\pay c,1;k}^1$ in \eqref{ddtPhikPayC}, we have that by H\"older inequality and Young's inequality,
\begin{align*}
T_{\pay c,1;k}^1
\leq \frac{1}{A^{5/4}}||\pa_y\wh{n}_k||_2^2+\frac{1}{A^{3/4}}||\pa_y\wh{c}_k||_2^2.
\end{align*}
By choosing $A$ large, these two terms can be absorbed by the negative terms in \eqref{ddtPhikN} and \eqref{ddtPhikPayC}. Combining the definition of $\al$ \eqref{def:abg}, H\"older inequality and Young's inequality, the $T_{\pay c, 1;k}^\al$ term in \eqref{ddtPhikPayC} can be estimated as follows,
\begin{align*}T_{\pay c, 1;k}^\al
\leq \frac{1}{A^{4/3}|k|^{1/3}}||\sqrt{\al}\pa_{yyy}\wh{c}_k||_2^2+\frac{1}{A^{4/3}|k|^{1/3}}||\pay \wh{n}_k||_2^2,
\end{align*}
which is consistent with \eqref{TimeDerivativeOfF} for $A$ large enough. For the first $\beta$ term in $T_{\pay c ,1;k}^\beta$, we can estimate it using the definition of $\beta$ \eqref{def:abg}, the fact that $||u'||_\infty\leq C$, H\"older inequality and Young's inequality as follows
\begin{align*}
2k Re\bigg\lan i\beta u'\frac{\pa_y\wh{n}_k}{A},\pa_{yy}\wh{c}_k\bigg\ran\lesssim \frac{1}{A^{4/3}}||\pa_y\wh{n}_k||_2^2+\frac{1}{A^{4/3}}||\pa_{yy}\wh{c}_k||_2^2.
\end{align*}
This term is consistent with \eqref{TimeDerivativeOfF} given $A$ chosen large. The second term in $T_{\pay c, 1;k}^\beta$ is the same as the first one through integration by part up to a controllable term, which can be estimated using the definition of $\beta$ \eqref{def:abg}, the fact that $||u''||_\infty\leq C$, H\"older inequality and Young's inequality as follows
\begin{align*}
-2kRe\bigg\lan i\beta u'' \pa_y\wh{c}_k,\frac{\pa_y\wh{n}_k}{A}\bigg\ran\lesssim \frac{1}{A^{4/3}}||\pa_y\wh{c}_k||_2^2+\frac{1}{A^{4/3}}||\pa_y\wh{n}_k||_2^2.
\end{align*}
As long as $A$ is large enough, these two terms can be absorbed by the negative terms in \eqref{ddtPhikN} and \eqref{ddtPhikPayC}.
Finally, for the $\gamma$ term $T_{\pay c,1;k}^\gamma$, we estimate it using 
the definition of $\gamma$ \eqref{def:abg}, $||u'||_{W^{1,\infty}}\leq C$, H\"older inequality and Young's inequality as follows
\begin{align*}T_{\pay c,1;k}^\gamma\lesssim \frac{||\pay\wh c_k||_2^2}{A^{2/3}}+\frac{||\pay\wh n_k||_2^2}{A^{4/3}}
.
\end{align*}
This is consistent with \eqref{TimeDerivativeOfF} given that $A$ is chosen large enough. The treatment of the term $T_{\pay c,3;k}^\beta$ in \eqref{ddtPhikPayC} is similar to the treatment of \eqref{Tc2kbeta}, so we omit the estimate for the sake of brevity. This  concludes the estimate of the time evolution $\frac{d}{dt}\Phi_k[\pa_yc_{\neq}]$.

The estimate of the time derivative $\frac{d}{dt}\Phi_k[\pa_xc_{\neq}]$ is similar to the estimates of the terms  $T_{\pay c,1;k}^{(\cdot)}$ and $T_{\pay c, 3;k }^\beta$ in \eqref{ddtPhikPayC}
, hence we omit it for the sake of brevity.

\subsection{Time evolution estimates: $\frac{d}{dt}\Phi[n_{\neq}]$}
\subsubsection{Estimate on the $L$ terms in \eqref{ddtPhikN}}
These terms are linear in the $k$-th mode, and it accordingly makes sense to estimate these terms $k$-by-$k$. In this subsection we prove that for $A$ sufficiently large, the $L_k^{(\cdot)}$ terms can be absorbed by the negative terms in the $\frac{d}{dt}\mathcal{F}$,  i.e.,
\bel\label{Lestimate}
L_k^1+L_k^\al+L_k^\beta+L_k^\gamma \leq -\frac{1}{4}\mathcal{N}_k[n_{\neq}]-\frac{1}{4}\mathcal{N}_k[\na c_{\neq}].
\eel
We start by estimating the $L_k^1$ term in \eqref{ddtPhikN}. We decompose it into two parts:
\begin{align}
L_k^1=2Re\big\lan -L_k ,{\widehat{n}_k}\big\ran=\frac{2}{A}Re\big\lan -\na\cdot(\na \wh{c}_kn_0)-\pay(\pay {c}_0 \wh{n}_k),{\widehat{n}_k}\big\ran=:L_{k,1}^1+L_{k,2}^1.
\end{align}
The term $L_{k,1}^1$ can be estimated as follows:
\begin{align*}
L_{k,1}^1\leq \frac{1}{AB}||\na \wh{n}_k||_2^2+\frac{B}{A}||\na\wh{c}_k||_2^2||n_0||_\infty^2.
\end{align*}
Thanks to the hypothesis \eqref{H4}, it is consistent with \eqref{Lestimate} if we choose $B$ then $A$ large enough.
The term $L_{k,2}^1$ is estimated using H\"older inequality, Young's inequality, Gagliardo-Nirenberg-Sobolev inequality and the chemical gradient $\pa_y c_0$ $L^4$ estimate \eqref{nasc0Lpestimate} as follows
\begin{align*}
L_{k,2}^1\leq\frac{1}{AB}||\pa_y\wh{n}_k||_2^2+\frac{B}{A}||\wh{n}_k||_4^2||\pa_y{c}_0||_4^2
\lesssim\frac{1}{AB}||\pa_y\wh{n}_k||_2^2+\frac{B}{A}||\wh{n}_k||_2^2||\pay{c}_0||_4^{8/3}
\lesssim\frac{1}{AB}||\pa_y\wh{n}_k||_2^2+\frac{B}{A}||\wh{n}_k||_2^2M^{8/3},
\end{align*}
which is consistent with \eqref{Lestimate} if $B$, then $A$, are chosen large.

Next, we decompose the $L_{k}^\al$ term into two parts:
\bel\label{Lk1al1Lk2al}
L_{k}^\al=-2Re\lan \al \pa_{yy}\widehat{n}_k,-L_k\ran=\frac{2}{A}Re\lan \al \pa_{yy}\widehat{n}_k,\na\cdot(\na \wh{c}_k n_0)+\pay(\pay c_0 \wh{n}_k)\ran=:L_{k,1}^\al+L_{k,2}^\al.
\eel
Combining the definition of $\al$ \eqref{def:abg}, the hypothesis \eqref{H3}, \eqref{H4}, H\"older inequality, Gagliardo-Nirenberg-Sobolev inequality and Young's inequality, the $L_{k,1}^\al$ term can be estimated as follows:
\begin{align*}L_{k,1}^\al=&\frac{2}{A}Re\lan \al\pa_{yy}\wh{n}_k,(\pa_{yy}-k^2)\wh{c}_k n_0\ran+\frac{2}{A}Re\lan \al\pa_{yy}\wh{n}_k,\pa_y \wh{c}_k\cdot \pa_y n_0\ran\\
\leq&\frac{1}{AB}||\sqrt{\al}\pa_{yy}\wh{n}_k||_2^2+\frac{B}{A^{5/3}}||\pa_{yy}\wh{c}_k||_2^2||n_0||_\infty^2+\frac{B|k|^4}{A^{5/3}}||\wh{c}_k||_2^2||n_0||_\infty^2+\frac{1}{BA}||\sqrt{\al}\pa_{yy}\wh{n}_k||_2^2+\frac{B\al}{A}||\pa_y\wh{c}_k||_\infty^2||\pay n_0||_2^2\\
\lesssim&\frac{1}{AB}||\sqrt{\al}\pa_{yy}\wh{n}_k||_2^2+\frac{B}{A^{5/3}}||\pay^2\wh{c}_k||_2^2C_{n,\infty}^2+\frac{B}{A^{5/3}}|k|^4||\wh{c}_k||_2^2C_{n,\infty}^2+\frac{||\pa_{yy} \wh{c}_k||_2^2}{A^{5/3}}+\frac{B^2}{A^{5/3}}||\pa_y\wh{c}_k||_2^2C_{n_0,\dot H^1}^4,
\end{align*}
which is consistent with the \eqref{TimeDerivativeOfF} if we choose $B$ then $A$ to be large enough.
For the $L_{k,2}^\al$ term in \eqref{Lk1al1Lk2al}, we can estimate them using the definition of $\al$ \eqref{def:abg}, hypothesis \eqref{H3}, Lemma \ref{nasc0Lpestimate}, Gagliardo-Nirenberg-Sobolev inequality, H\"older inequality and Young's inequality as follows:
\begin{align*}
L_{k,2}^\al=&\frac{2\sqrt{\ep_\al}}{A^{4/3}|k|^{1/3}}\lan \sqrt{\al}\pa_{yy}\wh{n}_k,\de_y c_0 \wh{n}_k+\pa_yc_0\pa_y\wh{n}_k\ran\\
\leq&\frac{1}{AB}||\sqrt{\al}\pa_{yy}\wh{n}_k||_2^2+\frac{B}{A^{5/3}|k|^{2/3}}||\de_y c_0||_4^2||\wh{n}_k||_4^2+\frac{B}{A^{5/3}|k|^{2/3}}||\pa_y c_0||_\infty^2||\pa_y \wh{n}_k||_2^2\\
\lesssim&\frac{1}{AB}||\sqrt{\al}\pa_{yy}\wh{n}_k||_2^2+\frac{B}{A^{5/3}|k|^{2/3}}||\de_y c_0||_4^2||\wh{n}_k||_2^{3/2}||\pa_y\wh{n}_k||_2^{1/2}+\frac{B}{A^{5/3}|k|^{2/3}}\sup_{0\leq s\leq t}||n_0(s)||_2^2||\pa_y\wh{n}_k||_2^{2}\\
\lesssim&\frac{1}{AB}||\sqrt{\al}\pa_{yy}\wh{n}_k||_2^2+\frac{B}{A|k|^{2/3}}||\de_y c_0||_4^2||\wh{n}_k||_2^2+\frac{B}{A^{11/3}|k|^{2/3}} ||\de_y c_0||_4^2  ||\pa_y\wh{n}_k||_2^2\\
&+\frac{B\CC^2}{A^{5/3}|k|^{2/3}}||\pa_y\wh{n}_k||_2^{2}\\
\lesssim&\frac{1}{AB}||\sqrt{\al}\pa_{yy}\wh{n}_k||_2^2+\frac{B}{A|k|^{2/3}}(C_{2,\infty}^2+C_{n_0,\dot H^1}^2)||\wh{n}_k||_2^2+\frac{B}{A^{5/3}|k|^{2/3}} (C_{2,\infty}^2+C_{n_0, \dot H^1}^2)  ||\pa_y\wh{n}_k||_2^2,
\end{align*}
which is consistent with \eqref{TimeDerivativeOfF} if we choose $B$ then $A$ large.

For the $L_{k}^\beta$ terms in \eqref{ddtPhikN}, we decompose it into four parts
\begin{align}
L_{k}^\beta=&-\frac{2k}{A}Re[\lan i\beta u'(\na\cdot(\na \wh{c}_k n_0)+\pay (\pay c_0\wh{n}_k)),\pa_y\widehat{n}_k\ran+\lan i\beta u'\widehat{n}_k,\pa_y (\na\cdot(\na \wh{c}_k n_0)+\pay (\pay c_0\wh{n}_k))\ran]\nonumber\\
=&L_{k,1}^\beta+L_{k,2}^\beta+L_{k,3}^\beta+L_{k,4}^\beta.
\label{Lkbetaterms}
\end{align}
The term $L_{k,1}^\beta$ can be estimated using the definition of $\al,\beta$ \eqref{def:abg}, hypothesis \eqref{H4}, the fact that $||u'||_{W^{1,\infty}}\leq C$, integration by parts, H\"older inequality and Young's inequality as follows:
\begin{align*}
L_{k,1}^\beta=&\frac{2k}{A}Re\lan i\beta u'' \pa_y \wh{c}_k n_0,\pa_y \wh{n}_k\ran +\frac{2k}{A}Re\lan i\beta u'\na \wh{c}_k n_0,\pa_{y}\na \wh{n}_k\ran\\
\lesssim& \frac{||\pa_y\wh{n}_k||_2^2}{A^{4/3}}+\frac{1}{A^{4/3}}||\pa_y\wh{c}_k||_2^2||n_0||_\infty^2+\frac{1}{AB}||\sqrt{\al}\pa_{yy}\wh{n}_k||_2^2+\frac{B}{A}||\pa_y\wh{c}_k||_2^2||n_0||_\infty^2\\
&+\frac{1}{AB}||\pa_y \wh{n}_k||_2^2+\frac{B}{A^{5/3}}|k|^{2}||k\wh{c}_k||_2^{2}||  n_0||_\infty^{2}\\
\lesssim& \frac{||\pa_y\wh{n}_k||_2^2}{A^{4/3}}+\frac{1}{A^{4/3}}||\pa_y\wh{c}_k||_2^2C_{n,\infty}^2+\frac{1}{AB}||\sqrt{\al}\pa_{yy}\wh{n}_k||_2^2+\frac{B}{A}||\pa_y\wh{c}_k||_2^2C_{n,\infty}^2\\
&+\frac{1}{AB}||\pa_y \wh{n}_k||_2^2+\frac{B}{A^{5/3}}|k|^{4}||\wh{c}_k||_2^{2}C_{n,\infty}^{2},
\end{align*}
which is consistent with \eqref{Lestimate} if we choose $B$ then $A$ large.
The term $L_{k,3}^\beta$ is the same as the $L_{k,1}^\beta$ term up to the following controllable term, which can be estimated using hypothesis \eqref{H4}, the fact that $||u'''||_\infty\leq C$, H\"older inequality and Young's inequality as follows
\begin{align*}
\frac{2k}{A}Re\lan i\beta u''\wh n_k,\na\cdot (\na \wh c_k n_0)\ran=&-\frac{2k}{A}Re\lan i\beta u'''\wh{n}_k,\pa_y\wh{c}_k n_0\ran-\frac{2k}{A}Re\lan i\beta u'' \na \wh n_k,\na\wh c_kn_0\ran\\
\leq &\frac{1}{A^{4/3}}||\wh{n}_k||_2^2+\frac{1}{A^{4/3}}||\pa_y \wh{c}_k||_2^2||n_0||_\infty^2+\frac{1}{A^{4/3}}||\na \wh n_k||_2^2+\frac{1}{A^{4/3}}||\na\wh c_k||_2^2||n_0||_\infty^2\\
\leq& \frac{1}{A^{4/3}}||\wh{n}_k||_2^2+\frac{1}{A^{4/3}}||\pa_y \wh{c}_k||_2^2C_{n,\infty}^2+\frac{1}{A^{4/3}}||\na \wh n_k||_2^2+\frac{1}{A^{4/3}}||\na\wh c_k||_2^2C_{n,\infty}^2,
\end{align*}
which is consistent with \eqref{TimeDerivativeOfF} if we choose $B$ then $A$ large.
For the $L_{k,2}^\beta$ term in \eqref{Lkbetaterms}, it can be estimated using integration by parts, hypothesis \eqref{H3}, the chemical gradient $\pa_y c_0$ $L^4$ estimate \eqref{nasc0Lpestimate}, the fact that $||u''||_\infty\leq C$, definition of $\al,\beta$ \eqref{def:abg}, H\"older inequality and Young's inequality as follows:
\begin{align*}
L_{k,2}^\beta
=&\frac{2k}{A}Re\lan i\beta u''\pay c_0\wh{n}_k,\pa_y\widehat{n}_k\ran+\frac{2k}{A}Re\lan i\beta u'\pay c_0\wh{n}_k,\pa_y\pay\widehat{n}_k\ran\\
\lesssim&\frac{1}{A^{4/3}B}||\pa_y \wh{n}_k||_2^2+\frac{B}{A^{4/3}}||\pa_y c_0||_4^2||\wh{n}_k||_2^{3/2}||\pa_y\wh{n}_k||_2^{1/2}\\
&+\frac{1}{AB}||\sqrt{\al}\pa_{yy} \wh{n}_k||_2^2+\frac{B|k|^{4/3}}{A^{2/3}}||\sqrt{\beta}u'\wh{n}_k||_2^2||\pay c_0||_\infty^{2}\\
\lesssim&\frac{1}{A^{4/3}B}||\pa_y \wh{n}_k||_2^2+\frac{B}{A^{4/3}}M^{8/3}||\wh{n}_k||_2^{2}+\frac{1}{AB}||\sqrt{\al}\pa_{yy} \wh{n}_k||_2^2+\frac{B|k|^{4/3}}{A^{2/3}}||\sqrt{\beta}u'\wh{n}_k||_2^2C_{n_0,L^2}^{2},
\end{align*}
which is consistent with \eqref{TimeDerivativeOfF} if we choose $B$ then $A$ large. The $L_{k,4}^\beta$ term is similar to the $L_{k,2}^\beta$ up to the following controllable term, which can be estimated using hypothesis \eqref{H3}, the chemical gradient $\pa_y c_0$ $L^4$ estimate \eqref{nasc0Lpestimate}, H\"older inequality, the boundedness of $||u'||_{W^{2,\infty}}$ and the definition of $\al,\beta$ \eqref{def:abg} as follows
\begin{align*}
\frac{2k}{A}Re\lan i\beta u''\wh n_k, \pay(\pay c_0 \wh n_k)\ran=&-\frac{2k}{A}Re\lan i\beta u'''\wh{n}_k,\pa_y c_0\wh{n}_k\ran-\frac{2k}{A}Re\lan i\beta u''\pay\wh n_k,\pay c_0\wh n_k\ran\\
\lesssim& \frac{||\wh{n}_k||_2^2||\pa_y c_0||_\infty}{A^{4/3}}+\frac{1}{A^{4/3}}||\pay \wh n_k||_2^2+\frac{1}{A^{4/3}}||\pay c_0||_4^{8/3}||\wh n_k||_2^2\\
\lesssim&\frac{||\wh{n}_k||_2^2C_{2, \infty}}{A^{4/3}}+\frac{1}{A^{4/3}}||\pay \wh n_k||_2^2+\frac{M^{8/3}}{A^{4/3}}||\wh n_k||_2^2. 
\end{align*}
It is consistent with \eqref{TimeDerivativeOfF} if we choose $B$ then $A$ large.

Finally, we decompose the $L_k^\gamma$ term in \eqref{ddtPhikN} as follows:
\bel\label{Lkgammaterms}
L_k^\gamma=-\frac{2|k|^2}{A}Re\lan \gamma (u')^2\widehat{n}_k,\na\cdot(\na \wh{c}_k n_0)+\pay(\pay c_0\wh{n}_k)\ran=: L_{k,1}^\gamma+L_{k,2}^\gamma.
\eel
The term $L_{k,1}^\gamma$ can be estimated using integration by parts, hypothesis \eqref{H3}, the fact that $||u'||_{W^{1,\infty}}\leq C$, definition of $\gamma$ \eqref{def:abg}, H\"older inequality and Young's inequality as follows:
\begin{align*}
L_{k,1}^\gamma
\lesssim&\bigg|\frac{2|k|^2}{A}Re\lan \gamma u'2 u''\wh{n}_k+\gamma u'^2\pa_y\wh{n}_k,\pa_y\wh{c}_k n_0\ran\bigg|+\frac{2|k|^2}{A}||\wh{n}_k||_2||\wh{c}_k||_2||n_0||_\infty\\
\lesssim&\frac{2}{A}||\wh{n}_k||_2^2+\frac{2}{A}||\pa_y\wh{c}_k||_2^2C_{n,\infty}^2+\frac{1}{AB}||\pay\wh{n}_k||_2^2+\frac{B}{A}||\pa_y\wh{c}_k||_2^2C_{n,\infty}^2+\frac{1}{AB}|k|^2||\wh{n}_k||_2^2+\frac{B}{A}||\wh{c}_k||_2^2|k|^2C_{n,\infty}^2,
\end{align*}
which is consistent with \eqref{TimeDerivativeOfF} if we choose $B$ then $A$ large.
The term $L_{k,2}^\gamma$ in \eqref{Lkgammaterms} can be estimated using integration by parts, hypothesis \eqref{H3}, the fact that $||u'||_{W^{1,\infty}}\leq C$, definition of $\gamma$ \eqref{def:abg}, H\"older inequality and Young's inequality as follows:
\begin{align*}
L_{k,2}^\gamma=&\frac{2|k|^2}{A}Re\lan \gamma u' u'' 2 \wh{n}_k+\gamma (u')^2\pa_y\wh{n}_k,\pa_y c_0 \wh{n}_k\ran\\
\lesssim &\frac{2}{A}||\wh{n}_k||_2^2||\pa_y c_0||_\infty+\frac{1}{AB}||\pa_y \wh{n}_k||_2^2+\frac{B}{A}||\wh{n}_k||_2^2||\pa_y c_0||_\infty^2\\
\lesssim &\frac{2}{A}||\wh{n}_k||_2^2C_{2,\infty}+\frac{1}{AB}||\pa_y \wh{n}_k||_2^2+\frac{B}{A}||\wh{n}_k||_2^2C_{2,\infty}^2,
\end{align*}
which is consistent with \eqref{TimeDerivativeOfF} if we choose $B$ then $A$ large.
\subsubsection{Estimate on $NL$ terms} \label{sec:NLED}
As these terms are nonlinear in non-zero frequencies, it is more natural to consider all of the frequencies at once. For the $NL_k^1$ term in \eqref{ddtPhikN}, we estimate it as follows
\begin{align*}
-\sum_{k \neq 0}2 Re\lan NL_k,\widehat{n}_k\ran = -2\brak{\frac{1}{A}\grad \cdot \left( n_{\neq} \grad c_{\neq}\right),n_{\neq}} = \frac{2}{A} \brak{ n_{\neq} \grad c_{\neq},\grad n_{\neq}}
\leq \frac{2}{A} \norm{\grad c_{\neq}}_{\infty }\norm{\grad n_{\neq}}_{2} \norm{n_{\neq}}_{2}.
\end{align*}
By hypothesis \eqref{H5}, for some constant $B > 0$,
\begin{align*}
-\sum_{k \neq 0}2 Re\lan NL_k,\widehat{n}_k\ran
& \lesssim \frac{1}{AB}\norm{\grad n_{\neq}}_{2}^2 + \frac{B}{A}\CC^{2}\norm{n_{\neq}}_{2}^2 .
\end{align*}
By first choosing $B$ large relative to the implicit constant, and then choosing $A$ large (relative to constants and $B$), these terms are absorbed by the negative terms in \eqref{ddtPhikN}.

For the $NL_k^{\al}$ term in \eqref{ddtPhikN}, we use the bootstrap hypotheses to deduce (using the definition of $\alpha$;  recall that $\alpha$ is a Fourier multiplier in $x$),
\begin{align*}
2Re \sum_{k \neq 0} \brak{\alpha(k) \partial_{yy} \widehat{n}_k, NL_k} & = \frac{2}{A} \brak{\alpha(\partial_x) \pa_{yy} n_{\neq}, \grad \cdot \left( n_{\neq} \grad c_{\neq}\right)} \\
& \lesssim \frac{1}{A^{4/3}} \norm{\sqrt{\al}\pa_{yy} n_{\neq}}_2\left(\norm{\grad n_{\neq}}_2 \norm{\grad c_{\neq}}_{\infty} + \norm{n_{\neq}}_\infty \norm{\na^2 c_{\neq}}_{2}\right) \\
& \lesssim \frac{1}{A^{4/3}} \norm{\sqrt{\al}\pa_{yy} n_{\neq}}_2^2+\frac{\CC^2}{A^{4/3}}\norm{\grad n_{\neq}}_2^2  + \frac{C_\infty^2}{A^{4/3}}\norm{\na^2c_{\neq}}_2^2,
\end{align*}
and choosing $A$ large, these terms are absorbed by the negative terms in \eqref{ddtPhikN}, \eqref{ddtPhikPaxC} and \eqref{ddtPhikPayC}.

There are two terms in $NL_k^{\beta}$ in \eqref{ddtPhikN}; we estimate the first as follows (using that $\beta(k) \lesssim A^{-1/3}\abs{k}^{-4/3}$ and defines a self-adjoint operator, chemical gradient $\pay c_0$ $L^p$ estimate \eqref{nasc0Lpestimate}, the fact that $||u'||_\infty\leq C$ and that $u$ does not depend on $x$):
\begin{align}
-2k\sum_{k \neq 0} Re\brak{i\beta(k) u'NL_k,\pa_y \widehat{n}_k} & =- \frac{2}{A}\brak{\beta(\partial_x) u' \pa_x\grad\cdot(n_{\neq} \grad c_{\neq}),\pa_y n_{\neq}} \nonumber \\
& \lesssim \frac{1}{A^{4/3}}\norm{\pa_y n_{\neq}}_2 (\norm{n_{\neq}}_\infty\norm{\na^2 c_{\neq}}_2+\norm{\na n_{\neq}}_2\norm{\na c_{\neq}}_\infty) \nonumber\\
& \lesssim  \frac{1}{AB}\norm{\pa_y n_{\neq}}_2^2+ \frac{B}{A^{5/3}}\norm{\na n_{\neq}}_2^2\norm{\na c_{\neq}}_\infty^{2}+ \frac{B}{A^{5/3}}\norm{ n_{\neq}}_\infty^{2}\norm{\na^2 c_{\neq}}_2^2. \label{firstbetaNL}
\end{align}
Recalling the bootstrap hypothesis, these terms are absorbed by the negative terms in \eqref{ddtPhikN}, \eqref{ddtPhikPayC} and \eqref{ddtPhikPaxC} given $B$ then $A$ large enough.
For the second term in $NL_k^{\beta}$ we use integration by parts to decompose it as follows
\begin{align*}
2Re\sum_{k \neq 0} \brak{\beta(k) u' ik\wh{n}_{k},\partial_y NL_k} & = \frac{2}{A}\brak{\beta(\partial_x) u'\pa_x n_{\neq}, \pa_y \grad\cdot(n_{\neq} \grad c_{\neq})}  \\
& = - \frac{2}{A}\brak{\beta(\partial_x) u'' \pa_x n_{\neq}, \grad\cdot(n_{\neq} \grad c_{\neq})} - \frac{2}{A}\brak{\beta(\partial_x) u' \pa_x\pa_y n_{\neq}, \grad\cdot(n_{\neq} \grad c_{\neq})}  \\
& = NL^\beta_{k,1} + NL^\beta_{k,2}.
\end{align*}
Using the definition of $\beta(\partial_x)$, the fact that $||u''||_\infty\leq C$ and that $u$ does not depend on $x$, we have,
\begin{align*}
\abs{NL^\beta_{k,1}} \lesssim& \frac{1}{A}\norm{n_{\neq}}_2^2+ \frac{1}{A^{5/3}}\norm{\na n_{\neq}}_2^2\norm{\na c_{\neq}}_\infty^{2}+ \frac{1}{A^{5/3}}\norm{ n_{\neq}}_\infty^{2}\norm{\na^2 c_{\neq}}_2^2.\end{align*}
yielding terms which are absorbed by the negative terms in \eqref{ddtPhikN}, \eqref{ddtPhikPaxC} and \eqref{ddtPhikPayC} for $A$ sufficiently large. The treatment of $ NL^\beta_{k,2}$ is similar to \eqref{firstbetaNL}, hence it is omitted for the sake of brevity.

Turn finally to term $NL_k^{\gamma}$ in \eqref{ddtPhikN} associated with $\gamma$:
\begin{align}
NL_k^{\gamma}
 & = -\frac{2}{A}\brak{\gamma(\partial_x) u' \partial_x n_{\neq}, u' \partial_x \grad\cdot(n_{\neq} \grad c_{\neq})} \nonumber\\
& = \frac{2}{A}\brak{\gamma(\partial_x) u' \partial_x \grad n_{\neq}, u' \partial_x(n_{\neq} \grad c_{\neq})} +\frac{4}{A}\brak{\gamma(\partial_x) u' u'' \partial_{x} n_{\neq}, \partial_x(n_{\neq} \partial_y c_{\neq})}\nonumber \\
&=:NL^\gamma_{k,1}+NL^\gamma_{k,2}.\label{NL gamma}
\end{align}
Then we use $\gamma(\partial_x) = \epsilon_\gamma \abs{\partial_x}^{-2}$ and interpolation to deduce the following bound for $NL^\gamma_{k,1}$:
\begin{align*}
NL^\gamma_{k,1}\lesssim&\frac{1}{A}||\sqrt{\gamma}u'\pa_x\na n_{\neq}||_2||\sqrt{\gamma}\pa_x(u'n_{\neq}\na c_{\neq})||_2\\
\lesssim&\frac{1}{A}||\sqrt{\gamma}u'\pa_x\na n_{\neq}||_2||u'n_{\neq}\na c_{\neq}||_2\\
\lesssim& \frac{1}{AB}\norm{\sqrt{\gamma} u' \partial_x \grad n_{\neq}}_2^2 + \frac{B}{A}\norm{ n_{\neq}}_\infty^2\norm{\na c_{\neq}}_2^2.
\end{align*}
For $B$, then $A$, chosen large, we may absorb these contributions in the negative terms in \eqref{ddtPhikN}, \eqref{ddtPhikPaxC} and \eqref{ddtPhikPayC}. Next we estimate the $NL^\gamma_{k,2}$ term in (\ref{NL gamma}) using the definition of $\gamma$ \eqref{def:abg}, the fact that $||u''||_\infty\leq C$ and the hypothesis \eqref{H5} as follows
\begin{align*}
NL^\gamma_{k,2}\lesssim&\frac{1}{A}||\sqrt{\gamma} u'|\pa_x|n_{\neq}||_2||n_{\neq}\pay c_{\neq}||_2 \lesssim \frac{1}{A}||\sqrt{\gamma} u'|\pa_x|n_{\neq}||_2^2+\frac{\CC^2}{A}||n_{\neq}||_2^2.
\end{align*}
Hence, for $A$ chosen large, we may absorb these contributions in the negative terms in \eqref{ddtPhikN}. This finishes the estimate of the $NL$ terms.

\section{Nonzero mode $L_t^2 \dot{H}_{x,y}^1$ estimate \eqref{ctrl:L2H1}}
The nonzero mode $L_t^2 \dot{H}_{x,y}^1$ estimate \eqref{ctrl:L2H1} comes from an estimate on the $\frac{d}{dt}||{n}_{\neq}||_2^2$ and the knowledge that $||{n}_{\neq}||_2^2$ is bounded by $4C_{ED}\left(||n_{in}||_{H^1}^2+1\right)$ from Hypothesis \eqref{H2}.
Indeed, from the nonzero mode equation \eqref{nonzeromodeN} and Lemma \ref{Lem:nasc0Lpestimate}, there holds for some universal constant $B$,
\begin{align}
\frac{1}{2}\frac{d}{dt}||{n}_{\neq}||_2^2=&\brak{n_{\neq},\frac{1}{A}\de n_{\neq}-\frac{1}{A}\pay(\pay c_0\cdot n_{\neq})-\frac{1}{A}\na\cdot(\na c_{\neq} n_0)-\frac{1}{A}(\na\cdot(\na c_{\neq} n_{\neq}))_{\neq}}\nonumber\\
\leq & -\frac{1}{2A}||\na n_{\neq}||_2^2 + \frac{4}{A}\norm{\pay c_0}_4^2\norm{n_{\neq}}_4^2+ \frac{4}{A}\norm{\grad c_{\neq}}_2^2\norm{n_0}_\infty^2 + \frac{4}{A} ||\na c_{\neq}||_2^2 ||n_{\neq}||_\infty^2 \nonumber \\
 \leq &-\frac{1}{2A}||\na n_{\neq}||_2^2+\frac{BM^4}{A}||n_{\neq}||_2^2 + \frac{BC_{n,\infty}^2}{A} ||\na c_{\neq}||_2^2 .
\label{d dt n neq 2 2}
\end{align}
Recalling hypothesis \eqref{H2},
\begin{equation}
\frac{1}{A}\int_0^{T_\star}\left(||n_{\neq}||_2^2+||\na c_{\neq}||_2^2\right)dt\leq \frac{1}{A}\int_0^{T_\star}C_{ED}\left(||n_{in}||_{H^1}^2+1\right)e^{-ct/A^{1/3}} dt\leq \frac{C}{A^{2/3}},
\end{equation}
which implies the following by inequality \eqref{d dt n neq 2 2} given $A$ large,
\begin{align}
\frac{1}{A}\int_0^{T_\star}||\na n_{\neq}||_2^2dt\leq \frac{1}{A^{1/3}}+2||n_{in}||_{2}^2\leq 4||n_{in}||_{2}^2.
\end{align}
As a result, we have proved \eqref{ctrl:L2H1}.

\section{Zero mode estimate \eqref{ctrl:ZeroMd}}\label{G(t)trick}
Before estimating the $L^2$ norm of the solution, we note that by non-negativity and the divergence structure of the equation \eqref{KS advection}, the $L^1_y$ norm of $n_0(y)$ is constant in time $\norm{n_0}_{L^1(\rr)} = \frac{1}{2\pi}\norm{n}_{L^1(\Torus\times\rr)} = \frac{M}{2\pi}$. We first estimate $||{n}_0||_2^2$, then estimate $||\pa_y {n}_0||_2^2$.
From equation \eqref{zeromode} we have, by Minkowski's inequality and Lemma \ref{Lem:nasc0Lpestimate},
\begin{align*}
\frac{1}{2}\frac{d}{dt}||n_0||_2^2 =&\left\lan n_0,\frac{1}{A}\pa_{yy}{n}_0-\frac{1}{A}\pa_y(\pa_y{c}_0{n}_0)-\frac{1}{A}(\na\cdot(\na c_{\neq} n_{\neq}))_0\right\ran\\
=&-\frac{1}{A}||\pa_y {n}_0||_2^2+\frac{1}{A}\lan \pa_y{n}_0,\pa_y {c}_0 n_0\ran+\frac{1}{A}\lan \pa_y {n}_0,(\pa_y c_{\neq} n_{\neq})_0\ran\\
\leq&-\frac{1}{2A}||\pa_yn_0||_2^2+\frac{1}{A}||\pa_y {c}_0||_4^2||{n}_0||_2^{3/2}||\pay n_0||_2^{1/2}+\frac{1}{A}||(\pa_y {c}_{\neq}{n}_{\neq})_0||_{L^2(\mathbb{R})}^2\\
\lesssim&-\frac{1}{4A}||\pa_y{n}_0||_2^2+\frac{ M^{8/3}}{A} ||{n}_0||_2^2+\frac{1}{A}||\pa_y {c}_{\neq}||_{L^2(\mathbb{T}\times \rr)}^2||{n}_{\neq}||_{L^\infty(\mathbb{T}\times \rr)}^2.\\
\end{align*}
Recall the following Nash inequality on $\mathbb{R}$:
\bel\label{Nash ineq}
||\rho||_{L^2(\mathbb{R})}\lesssim ||\rho||_{L^1(\mathbb{R})}^{2/3}||\pa_y\rho||_{L^2(\mathbb{R})}^{1/3}.
\eel
Hence, by setting $\rho=n_0$, we have 
\be
-||\pa_y{n}_0||_2^2\leq-\frac{||{n}_0||_2^6}{C||n||_1^4}\leq -\frac{||{n}_0||_2^6}{CM^4}.
\ee
Combining this with the time evolution of $||n_0||_2^2$, we obtain
\begin{align}\label{n0L22_1}
\frac{d}{dt}|| {n}_0||_2^2\lesssim-\frac{1}{A}\frac{||{n}_0||_2^2}{M^4}\bigg(\frac{||{n}_0||_2^4}{C}-M^{20/3}\bigg)+\frac{1}{A}||\pa_y {c}_{\neq}||_{L^2(\mathbb{T}\times \rr)}^2||{n}_{\neq}||_{L^\infty(\mathbb{T}\times \rr)}^2 
.\end{align}
Define the following quantity $G$ to be
\begin{align}
G(t):=&\int_0^t\frac{B}{A}||\pay c_{\neq}||_2^2||n_{\neq}||_\infty^2d\tau, \quad \forall t\geq 0. \label{def:Gfirst}
\end{align}
By the bootstrap hypotheses,
\begin{equation}
G(t)\leq \int_0^t\frac{B}{A}C_{ED}(||n_{in}||_{H^1}^2+1)e^{-\eta\tau/A^{1/3}}C_{n,\infty}^2d\tau\leq \frac{C}{A^{2/3}},
\end{equation}
hence $G\lesssim 1$ given $A$ large.
Applying this in \eqref{n0L22_1}, we have
\be\ba
\frac{d}{dt}(||n_0||_2^2-G(t))\lesssim &-\frac{1}{A}\frac{||n_0||_2^2}{M^4}\left(\frac{||n_0||_2^4}{C}-M^{20/3}\right)\\
\lesssim&-\frac{1}{A}\frac{||n_0||_2^2}{M^4}\left(\frac{||n_0||_2^2-G(t)}{\sqrt{C}}-M^{10/3}\right)\left(\frac{||n_0||_2^2}{\sqrt{C}}+M^{10/3}\right).
\ea\ee
Choosing $A$ large relative to $||n_{in}||_{H^1}^2$, $\CC^2$ and universal constants, we have
\begin{align}
||n_0 ||_2^2& \lesssim G(t)+M^{10/3}+\norm{n_{in}}_2^2\nonumber \\
& \leq C\left( 1+  ||n_{in}||_2^2 +M^{10/3}\right)\nonumber\\
& =: C_{n_0,L^2}^2(||n_{in}||_2^2,M).\label{ML2}
\end{align}
This completes the estimate on $\norm{n_0}_2$. Combining with Lemma \ref{nasc0Lpestimate} and adjusting the constant $C_{n_0, L^2}$ defined in \eqref{ML2} yield the first estimate in conclusion (\ref{ctrl:ZeroMd}).

Before estimating $||\pa_y n_0||_2$, we first note the following estimate on $\int_0^{T_\star}||\na^2 c_{\neq}||_2^2dt$ from hypothesis \eqref{H2} and \eqref{ddtPhikPayC}, \eqref{ddtPhikPaxC},
\begin{align}\label{nanac}
\frac{1}{8A}\int_0^{T_\star}||\na^2c_{\neq}||_2^2dt\leq\mathcal{F}(0)-\mathcal{F}(T_\star)\leq C_{ED}\left(\norm{n_{in}}_{H^1}^2+1\right).
\end{align}
Now we estimate the $||\pa_y n_0||_2$. Estimating the time evolution of $||\pa_y n_0||_2$ using Gagliardo-Nirenberg-Sobolev inequality, Young's inequality, Minkowski inequality, Lemma \ref{Lem:nasc0Lpestimate} and the time integral estimate of $||\na^2 c_{\neq}||_2^2$ \eqref{nanac}, we have that 
\begin{align}
\frac{1}{2}&\frac{d}{dt}||\pa_y n_0||_2^2\nonumber\\
\leq &-\frac{||\pa_y^2 n_0||_2^2}{2A}+\frac{4||\pa_y^2 c_0||_2^2||n_0||_\infty^2}{A}+\frac{4||\pa_y c_0||_4^2||\pa_y n_0||_4^2}{A}+\frac{4}{A}||\pay(\pa_y c_{\neq}n_{\neq})_0||_2^2\nonumber\\
\leq&-\frac{||\pa_y^2 n_0||_2^2}{2A}+\frac{4C(\sup_{0\leq s\leq t}||\pa_y n_0(s)||_1^2+||\pay^2( c_{in})_0||_2^2)||n_0||_2||\pay n_0||_2}{A}+\frac{CM^2|| n_0||_2^{3/4}||\pa_y^2 n_0||_2^{5/4}}{A}\nonumber\\
&+\frac{4}{A}||\pa_y^2 c_{\neq}||_{L^2(\Torus\times \rr)}^2||n_{\neq}||_{L^\infty(\Torus\times \rr)}^2+\frac{4}{A}||\pa_y c_{\neq}||_{L^\infty(\Torus\times\rr)}^2||\pay n_{\neq}||_{L^2(\Torus\times \rr)}^2.\label{time evolution of na f L2}
\end{align}
We define
\begin{equation}\label{G2}
G(t):=\frac{4}{A}\int_0^{t}||\pa_y^2 c_{\neq}||_2^2||n_{\neq}||_{\infty}^2+||\pa_y c_{\neq}||_\infty^2||\pa_yn_{\neq}||_{2}^2ds.
\end{equation}
Note that by the hypothesis \eqref{H1} and time integral control \eqref{nanac}, we have that
\begin{align}
 G(t)\lesssim C_{2,\infty}^4
\end{align}
for all $t\leq T_\star$.
Now we apply Gagliardo-Nirenberg-Sobolev inequality, Lemma \ref{Lem:nasc0Lpestimate} and definition of $G(t)$ \eqref{G2} to rewrite the inequality \eqref{time evolution of na f L2} as follows:
\begin{align}
\frac{d}{dt}\left(||\pa_y n_0||_2^2-G(t)\right)\lesssim &\frac{1}{A}\bigg(-\frac{||\pa_y n_0||_2^2(||\pa_y n_0||_2^2-G)}{4CC_{n_0,L^2}^2}
+\CC^4||\pa_y n_0||_2^2\nonumber\\
&+\sup_{0\leq \tau\leq t}(||\pa_y n_0(\tau)||_2^2-G(\tau))+\sup_{0\leq\tau\leq t} G(\tau)+(M^{16/3}+1)C_{n_0,L^2}^{2}\bigg).\label{f 0 H1 0}
\end{align}
Now because $G(t)\lesssim C_{2,\infty}^4$, by a comparison principle, we can prove that
\begin{align}\label{f 0 H1}
||\pa_y n_0||_2\leq 2C_{n_0,\dot H^1},
\end{align}
where $C_{n_0,\dot H^1}$ is chosen properly.
The reasoning is as follows. Since we can set the $C_{H^1}'$ large such that $||\pa_y (n_{in})_0||_2^2\ll C_{H^1}'^2$, $||\pa_y n_0||_2^2-G(t)$ will reach the value $C_{H^1}'^2$ at the first time $t_\star>0$. At time $t_\star$, we have that $\sup_{0\leq\tau\leq t_\star}\left(||\pa_y n_0(\tau)||_2^2-G(\tau)\right)=||\pa_y n_0(t_\star)||_2^2-G(t_\star)=C_{H^1}'^2$. Combining this fact and the differential inequality \eqref{f 0 H1 0} yields
\begin{align*}
\frac{d}{dt}\left(||\pa_y n_0||_2^2-G\right)\bigg|_{t=t_\star}\lesssim &\frac{1}{A}\bigg(-\frac{||\pa_y n_0(t_\star)||_2^2C_{H^1}^{'2}}{4CC_{n_0,L^2}^2}
+\CC^4||\pa_y n_0(t_\star)||_2^2\\
&+||\pa_y n_0(t_\star)||_2^2-G(t_\star)+C_{2,\infty}^4+M^{16/3}C_{n_0,L^2}^{2}\bigg)\\
\lesssim&\frac{1}{A}\left(-\frac{(C_{H^1}')^4}{4C C_{n_0, L^2}^2}+\CC^4(C_{H^1}')^2+\CC^8\right)\\
<&0.
\end{align*}
The last inequality $<0$ is true if we pick $C_{H^1}'$ large enough. On the other hand, $\frac{d}{dt}\left(||\pa_y n_0||_2^2-G\right)\bigg|_{t=t_\star}\geq 0$ at the first break through time $t_\star$. As a result, we reach a contradiction. Therefore, we have that
\begin{align*}
||\pay n_0||_{L^\infty_t(0,T_\star; L^2_y)}^2\leq sup_{0\leq t\leq T_\star} G(t)+C_{H^1}'^2\leq BC_{2,\infty}^6.
\end{align*}
Now we just need to choose $C_{n_0,\dot H^1}^2$ much larger than the right hand side to conclude the proof of \eqref{ctrl:ZeroMd}.

\section{Uniform $L^\infty$ control \eqref{ctrl:n_Linf} and \eqref{ctrl:na_c_Linf}}
In this section we prove the uniform $L^\infty$ control \eqref{ctrl:n_Linf} and \eqref{ctrl:na_c_Linf}. We separate the proof into two different time regimes, namely, the initial time $t\leq A^{1/3+\epsilon}$ and the long time $t\geq A^{1/3+\epsilon}$. Here $\epsilon>0$ is a small constant determined by the proof.  For the sake of clarity, we use $C_{n,\infty}^{in}, \enskip C_{\na c_{\neq},\infty}^{in}$ to denote bounds in the initial time and $C_{n,\infty}^{long}, \enskip C_{\na c_{\neq},\infty}^{long}$ to denote bounds in the long time. At the end of the proof, we will take the $C_{n,\infty}$ to be large compared to $C_{n,\infty}^{in}$ and $C_{n,\infty}^{long}$ and take the $C_{\na c_{\neq},\infty}$ large compared to $C_{\na c_{\neq},\infty}^{in}$ and $C_{\na c_{\neq},\infty}^{long}$.

\subsection{Initial Time Layer Estimate}
In this subsection, we would like to prove the following lemma:
\begin{lem}
Under the assumptions of Proposition \ref{prop:boot2D}, there exist a constant $0<\epsilon<\frac{1}{12}$ independent of the solution and constants $C_{n,\infty},C_{\na c_{\neq},\infty},C_{\pa_xn,\infty}$ depending on $C_{ED}, n_{in}, M$ such that the following estimates hold on the time interval $0\leq t\leq A^{1/3+\ep}$ when $A$ is chosen large enough:\begin{subequations}\label{ctrl_in}
\begin{align}
||n(t)||_{{\infty}}\leq &C_{n,\infty}^{in}(n_{in}, C_{ED},M);\label{ctrl:n_Linf_in}\\
||\na c_{\neq}(t)||_{{\infty}}\leq& C_{\na c_{\neq},\infty}^{in}(n_{in}, C_{ED}, M);\label{ctrl:na_c_Linf_in}\\
||\pa_x n(t)||_\infty\leq &C_{\pa_xn,\infty}(||n_{in}||_{H^1}),\quad \forall t\in[0,A^{1/3+\ep}].\label{Conclusion_on_paxn}
\end{align}
\end{subequations}
\end{lem}
\begin{rmk}
In the proof of the lemma, the shear flow destabilizing effect has to be treated carefully because the enhanced dissipation effect of the shear flow is too weak at the initial time. We will propagate the estimates \eqref{ctrl_in} till $t=A^{1/3+\ep}$. After this time threshold, the enhanced dissipation kicks in to stabilize the dynamics.
\end{rmk}
\begin{proof}
We use a bootstrap argument to prove the lemma. Assume that for constants $C_{n,\infty}^{in}, \enskip C_{\na c_{\neq},\infty}^{in},\enskip C_{\pa_x n,\infty}$ depending on the proof, $T_{\star\star}\in [0,A^{1/3+\ep}]$ is the maximal time on which the following hypothesis is satisfied:\begin{subequations}\label{H_in}
\begin{align}
||n(t)||_{{\infty}}\leq &2C_{n,\infty}^{in};\label{Hyp:n_Linf_in}\\
||\na c_{\neq}(t)||_{{\infty}}\leq& 2C_{\na c_{\neq},\infty}^{in};\label{Hyp:na_c_Linf_in}\\
||\pa_x n(t)||_\infty\leq &2C_{\pa_x n,\infty},\quad \forall t\in [0, T_{\star\star}), T_{\star\star}\leq \min\{A^{1/3+\ep},T_\star\}.\label{Hyp_on_paxn}
\end{align}\end{subequations}
We will show that all the estimates \eqref{H_in} hold on the same time interval $[0,T_\star]$ with '1' instead of '2' if we choose $A_0$ large. These improvements combined with the local well-posedness of the equation \eqref{KS advection} yield \eqref{ctrl_in}. 


We split the proof into three steps. In the first step, we obtain the improvement to \eqref{Hyp:n_Linf_in} together with a suboptimal estimate of $||\na c_{\neq}||_p,\forall p<\infty$. Here the estimate in $||\na c_{\neq}||_p,\forall p<\infty$ is suboptimal in the sense that on the interval $[0,T_{\star\star})$, the estimate loses a small power of $A$, i.e., $||\na c_{\neq}||_p\lesssim A^{\delta},\delta>0$. In order to compensate for the loss in powers of $A$, we need information about the higher regularity of $n_{\neq}$. This is why we propagate another estimate \eqref{Conclusion_on_paxn} in the initial time layer $[0,T_{\star\star})$. In the second step, we complete the proof of \eqref{Conclusion_on_paxn}. In the last step, we use the extra regularity information to get the optimal $L^\infty$ bound of $\na c_{\neq}$.

\textbf{First step:} We prove the improvement to \eqref{Hyp:n_Linf_in} on $[0,T_{\star\star})$. We start with the estimate on $||\pa_x c_{\neq}||_4$. Direct energy estimate yields
\begin{align}\label{dt_pax_c_L4}
\frac{d}{dt}||\pa_x c_{\neq}||_4^4\leq-\frac{3}{2A}||\na (\pa_x c_{\neq})^2||_2^2+\frac{6}{A}||\pa_x c_{\neq}||_4^2||n_{\neq}||_4^2-\frac{4}{A}||\pa_x c_{\neq}||_4^4.
\end{align}
Integration in time yields 
\begin{equation}\label{pax_cneq_L4_control}
||\pa_x c_{\neq}(t)||_4\leq\sqrt{3}\frac{\sqrt{t}}{A^{1/2}}\sup_{0\leq s \leq t}||n_{\neq}(s)||_4+||\pa_x(c_{in})_{\neq}||_4.
\end{equation}

With the equation \eqref{nonzeromodeC}
, we estimate the time evolution of the $L^4$ norm of $\pa_y c_{\neq}$:
\begin{align*}
\frac{d}{dt}||\pa_y c_{\neq}||_4^4\leq -\frac{3}{2A}||\na(\pa_y c_{\neq})^2||_2^2+\frac{6||n_{\neq}||_4^2||\pa_y c_{\neq}||_4^2}{A}+4||\pa_y c_{\neq}||_4^3||u'\pa_x c_{\neq}||_4-\frac{4}{A}||\pa_y c_{\neq}||_4^4.
\end{align*}
As in the $\pa_x c_{\neq}$ case, we drop the negative term at the moment, and end up with the following inequality
\begin{eqnarray}\label{pa_y_c_neq_4}
\frac{d}{dt}||\pa_yc_{\neq}||_4^4\leq \frac{6||n_{\neq}||_4^2||\pa_y c_{\neq}||_4^2}{A}+4||\pa_y c_{\neq}||_4^3||u'\pa_x c_{\neq}||_4.
\end{eqnarray}
Now the idea is to compare $||\pa_y c_{\neq}||_4$ with the solution to the following differential equation,
\begin{eqnarray}
\frac{d}{dt}f^4= 10f^3\left(||u'\pa_x c_{\neq}||_4+\frac{||n_{\neq}||_4^2}{A}+\frac{1}{A}\right),\label{f_1}\\
f(0)=1>CA^{-q}\geq||\pa_y (c_{in})_{\neq}||_4.
\end{eqnarray}
and show that $||\pa_y c_{\neq}(t)||_4\leq f(t)$ for $t\leq T_{\star\star}$. The function f is estimated using \eqref{pax_cneq_L4_control} and the fact $q>1/2$ as follows:
\begin{align}
f(t)\lesssim&1+\frac{1}{A^{1/2}}+\int_0^t||u'\pa_x c_{\neq}(s)||_4+\frac{||n_{\neq}(s)||_4^2}{A}ds\nonumber\\
\lesssim& 1+A^{1/2-q}+\frac{t^{3/2}}{A^{1/2}}\sup_{0\leq s\leq t}||n_{\neq}(s)||_4+\frac{1}{A^{1/2}}\sup_{0\leq s\leq t}||n_{\neq}(s)||_4^2, \quad\forall t\leq A^{1/3+\ep}.\label{f_ctrl}
\end{align}
Next we show that $||\pa_y c_{\neq}||_4\leq f$ for $\forall t\in[0,T_{\star\star})$. Since $f$ is strictly increasing in time, $f\geq 1$. Assume that there exists a first time $t^\star\leq T_{\star\star}$ such that $||\pa_y c_{\neq}(t^\star)||_4^4$ is equal to the function $f^4(t^\star)$. At time $t^\star$, we have $||\pa_y c_{\neq}(t^\star)||_4=f(t^\star)\geq 1$, which yields the following relation
\begin{equation}||\pa_y c_{\neq}(t^\star)||_4^3\geq||\pa_y c_{\neq}(t^\star)||_4^2.
\end{equation}
Combining this with \eqref{pa_y_c_neq_4}, \eqref{f_1} yields that at time $t^\star$,
\begin{equation}
\frac{d}{dt}||\pa_y c_{\neq}||_4^4\bigg|_{t=t^\star}\leq \left(\frac{6||n_{\neq}||_4^2}{A}+4||u'\pa_x c_{\neq}||_4\right)||\pa_y c_{\neq}||_4^3\bigg|_{t=t^\star}<\frac{d}{dt}f^4\bigg|_{t=t^\star}.
\end{equation}
On the other hand, $\frac{d}{dt}||\pa_yc_{\neq}||_4^4\bigg|_{t=t^\star}\geq \frac{d}{dt}f^4\bigg|_{t=t^\star}$ at the first break-through time $t^\star$, which is a contradiction. As a result, we have that $||\pa_y c_{\neq}(t)||_4\leq f(t), \enskip\forall t\leq T_{\star\star}$, which together with \eqref{f_ctrl} yields the following estimate
\begin{equation}\label{paycneq_L4_ctrl}
||\pa_y c_{\neq}(t)||_4\lesssim 1+\frac{t^{3/2}}{A^{1/2}}\sup_{0\leq s\leq t}||n_{\neq}(s)||_4+\frac{1}{A^{1/2}}\sup_{0\leq s\leq t}||n_{\neq}(s)||_4^2,\quad \forall t\leq T_{\star\star}.
\end{equation}
Combining Lemma \ref{nasc0Lpestimate}, $||\na c_{\neq}||_4$ estimates \eqref{pax_cneq_L4_control} and \eqref{paycneq_L4_ctrl}, we  estimate the time evolution of $||n||_4^4$ as follows
\begin{align*}
\frac{d}{dt}||n||_4^4\lesssim&-\frac{3}{2A}||\na(n^2)||_2^2+\frac{||\na (n^2)||_2^{3/2}||n^2||_2^{1/2}(||\na c_{\neq}||_4+||\pay c_0||_4)}{A}\\
\lesssim&\frac{||n||_4^4}{A}\bigg(1+M^4+||\na(c_{in})_0||_4^4+\frac{t^2}{A^{2}}\sup_{0\leq s\leq t}||n_{\neq}(s)||_4^4+\frac{t^{6}}{A^{2}}\sup_{0\leq s\leq t}||n_{\neq}(s)||_4^4+\frac{1}{A^{2}}\sup_{0\leq s\leq t}||n_{\neq}(s)||_4^8\bigg).
\end{align*}
Thanks to the hypothesis \eqref{H4}, conservation of mass and H\"older inequality, we can take $A$ large enough such that the above estimate can be simplified as follows:
\begin{align*}
\frac{d}{dt}||n||_4^4\leq\frac{C||n||_4^4}{A}(M^4+1+||\pay (c_{in})_0||_4^4+A^{6\ep}\sup_{0\leq s\leq t}||n(s)||_4^4).
\end{align*}
Now we can compare the $||n||_4^4$ to the solution to the following differential equation:
\begin{align*}
\frac{d}{dt}f=\frac{2Cf}{A}(M^4+||\pay (c_{in})_0||_4^4+1+A^{6\ep}f),\quad f(0)>\max\{1, ||n_{in}||_4^4\}.
\end{align*}
The strictly increasing solution $f$ is bounded $f\leq C(n_{in})$ on the interval $[0,A^{1/3+\ep}]$ if $\ep$ is chosen small enough and $A$ is chosen large enough compared to $M, \enskip ||\pay(c_{in})_0||_4$ and $C$. Assume that there exists a first time $0<t_\star\leq A^{1/3+\ep}$ such that $||n(t_\star)||_4^4$ is equal to the function $f(t_\star)$. Since $f$ is strictly increasing, at the first break-through time $t_\star$, we have $||n(t_\star)||_4=\sup_{0\leq s\leq t_\star}||n(s)||_4$, which yields the following relation
\begin{equation}
\frac{d}{dt}||n||_4^4\bigg|_{t=t_{\star}}\leq\frac{C||n||_4^4}{A}(M^4+||\pay( c_{in})_0||_4^4+1+A^{6\ep}||n||_4^4)\bigg|_{t=t_\star}<\frac{d}{dt}f\bigg|_{t=t_\star}.
\end{equation}
On the other hand, $\frac{d}{dt}||n||_4^4\bigg|_{t=t_\star}\geq \frac{d}{dt}f\bigg|_{t=t_\star}$ at the first break-through time $t_\star>0$, which is a contradiction. As a result, we have that
\begin{equation}\label{n_4_bound_initial_layer}
||n(t)||_4\leq C_{n,L^4}^{in}(n_{in}), \quad\forall t\in [0, T_{\star\star}).
\end{equation}

Next we start the iteration process. Assume that $||n||_{p}$ is bounded, we estimate the $||n||_{2p}$ in terms of $||n||_p$. We start with estimating the $||\pa_x c_{\neq}||_{2p}^{2p}$. By calculating the time derivative, we see that
\begin{align*}
\frac{1}{2p}\frac{d}{dt}||\pa_x c_{\neq}||_{2p}^{2p}=&-\frac{2p-1}{Ap^2}||\na(\pa_xc_{\neq})^{p}||_2^2+\frac{2p-1}{Ap}||\na(\pa_x c_{\neq})^p||_2||(\pa_x c_{\neq})^{p-1}n_{\neq}||_2-\frac{1}{A}||\pa_x c_{\neq}||_{2p}^{2p}\\
\leq&-\frac{2p-1}{2Ap^2}||\na(\pa_xc_{\neq})^{p}||_2^2+\frac{p}{A}||\pa_x c_{\neq}||^{2p-2}_{2p}||n_{\neq}||_{2p}^2-\frac{1}{A}||\pa_x c_{\neq}||_{2p}^{2p}.
\end{align*}
As a result, we have that
\begin{equation}
\frac{d}{dt}||\pa_x c_{\neq}||_{2p}^2\leq \frac{2p}{A}||n_{\neq}||_{2p}^{2},
\end{equation}
which yields
\begin{equation}\label{pa_x_c_neq_2p}
||\pa_x c_{\neq}(t)||_{2p}\lesssim \sqrt{p}\sup_{0\leq s\leq t}||n_{\neq}(s)||_{2p}\frac{t^{1/2}}{A^{1/2}}+||\pa_x(c_{in})_{\neq}||_{2p},\quad \forall t\in [0,A^{1/3+\ep}].
\end{equation}
Next we estimate the time evolution of $||\pa_y c_{\neq}||_{2p}^{2p}$,
\begin{equation}
\frac{d}{dt}||\pa_y c_{\neq}||_{2p}^{2p}\leq 2p||u'\pa_x c_{\neq}||_{2p}||\pa_y c_{\neq}||_{2p}^{2p-1}+\frac{2p^2}{A}||\pa_y c_{\neq}||^{2p-2}_{2p}||n_{\neq}||_{2p}^2-\frac{2p}{A}||\pa_y c_{\neq}||_{2p}^{2p}.
\end{equation}
By comparing the solution with the following strictly increasing function $f$
\begin{equation}
\frac{d}{dt}f^{2p}= 4pf^{2p-1}\left(||u'\pa_x c_{\neq}||_{2p}+p\frac{||n_{\neq}||_{2p}^2}{A}+\frac{1}{A}\right),\quad f(0)=1>CA^{-q}\geq||\pa_y (c_{in})_{\neq}||_{2p},
\end{equation}
and applying a similar argument to prove \eqref{n_4_bound_initial_layer}, we have that
\begin{equation}\label{pa_y_c_neq_2p}
||\pa_y c_{\neq}(t)||_{2p}\leq f(t)\lesssim 1 + \frac{t^{3/2}}{A^{1/2}}\sqrt{p}\sup_{0\leq s\leq t}||n_{\neq}(s)||_{2p}+A^{-q+1/3+\ep}+p\frac{\sup_{0\leq s\leq t}||n_{\neq}(s)||_{2p}^2}{A^{1/2}}, \quad\forall t\in [0,T_{\star\star}].
\end{equation}

Next we estimate the time evolution of $||n||_{2p}^{2p}$. Applying the hypothesis, $||\na c_{\neq}||_{2p}$ estimates \eqref{pa_y_c_neq_2p}, \eqref{pa_x_c_neq_2p} and Lemma \ref{nasc0Lpestimate}, we have the following estimate by picking $A$ large
\begin{align*}
\frac{1}{2p}\frac{d}{dt}||n||_{2p}^{2p}=&-\frac{2p-1}{Ap^2}||\na(n^{p})||_2^2+\frac{2p-1}{Ap}||\na(n^p)||_2||n^p \na c||_2\\
\leq&-\frac{2p-1}{Ap^2}||\na(n^{p})||_2^2+\frac{2p-1}{Ap}||\na(n^p)||_2||n^p||_2^{1-1/p}||n||_\infty|| \na c||_{2p}\\
\leq&-\frac{2p-1}{2Ap^2}||\na(n^{p})||_2^2+\frac{Cp^3}{A}||n^p||_1^{\frac{2p-2}{p+1}}||n||_\infty^{\frac{4p}{p+1}}|| \na c||_{2p}^{\frac{4p}{p+1}}\\
\leq&\frac{Cp^7}{A}||n^p||_1^{\frac{2p-2}{p+1}}C_{2,\infty}^{4}\bigg(C(M, C_{n_0, L^2},\pa_y (c_{in})_0)+A^{3\ep/2}\sup_{0\leq s\leq t}||n_{\neq}(s)||_{2p}\bigg)^{4}
\end{align*}
Time integrating on both side of the estimate and applying the hypothesis \eqref{H4}, conservation of mass and H\"older inequality, we have
\begin{align}\label{n_2p_initial_layer}
\sup_{0\leq s\leq T_{\star\star}}||n(s)||_{2p}^{2p}\leq \frac{p^8}{A^{2/3-7\ep}}C(C_{2,\infty},\pa_y (c_{in})_0)\sup_{0\leq s\leq T_{\star\star}}||n(s)||_{p}^{2p(\frac{p-1}{p+1})}+||n_{in}||_{2p}^{2p}.
\end{align}

Finally, we use the \eqref{n_2p_initial_layer} together with \eqref{n_4_bound_initial_layer} to prove the $||n||_{L^\infty(0, T_{\star\star};L^\infty)}\leq  C_{n,\infty}^{in}$. Note that if for $\forall j\in \mathbb{N}$, $\sup_{0\leq s<T_{\star\star}}||n(s)||_{2^j}\leq 1$, we have that $\sup_{0\leq s<T_{\star\star}}||n||_\infty\leq 1$, and the result follows. Therefore, we define $4<p_\star=2^{j_\star}\in 2^\mathbb{Z}$ to be the first integer such that $\sup_{0\leq s<T_{\star\star}}||n||_{p_\star}\geq 1$. Note that for $p={p_\star/2}$, 
\begin{equation}\label{nLpstar/2}
||n||_{L^\infty_t(0, T_{\star\star};L_{x,y}^{p_\star/2})}\leq \max\{C_{n,L^4}^{in},1\}. 
\end{equation}In the following argument, we will only care about $p>p_\star$ since we want to find the limit of $||n||_{L^\infty_t(0, T_{\star\star};L_{x,y}^p)}$ as $p\rightarrow\infty$.

By the H\"{o}lder's inequality,
\begin{equation*}
1\leq \sup_{0\leq s\leq T_{\star\star}}||n(s)||_{p_*}\leq ||n(t)||_1^{\theta}\sup_{0\leq s\leq T_{\star\star}}||n(s)||_{p}^{1-\theta}, \quad \forall p>p_\star.
\end{equation*}
Combining this estimate with the conservation of mass, we can get a lower bound for $\sup_{0\leq s\leq T_{\star\star}}||n(s)||_{p}$
\begin{equation}
\sup_{0\leq s\leq T_{\star\star}}||n(s)||_{p}\geq (1+M)^{-\frac{\theta}{1-\theta}}\geq(1+M)^{-2/5}, \quad\forall p=2^j\in 2^{\mathbb{N}}, j> j_\star>2.
\end{equation}
Combining this with \eqref{n_2p_initial_layer}, we have that
\begin{eqnarray}
\sup_{0\leq s\leq T_{\star\star}}||n(s)||_{2p}^{2p}\leq \frac{p^8}{A^{2/3-7\ep}}C(C_{2,\infty})\sup_{0\leq s\leq T_{\star\star}}||n(s)||_{p}^{2p}+||n_{in}||_{2p}^{2p}.
\end{eqnarray}
Now we can pick the $A$ big such that
\begin{equation}
\sup_{0\leq s\leq T_{\star\star}}||n(s)||_{2p}^{2p}\leq p^8\sup_{0\leq s\leq T_{\star\star}}||n(s)||_{p}^{2p}+||n_{in}||_{2p}^{2p},\quad \forall p=2^j\geq p_\star,j\in \mathbb{N}.
\end{equation}
Now by the $L^4$ bound of $n$ \eqref{n_4_bound_initial_layer}, the $L^{p_\star/2}$ bound of $n$ \eqref{nLpstar/2} and the standard Moser-Alikakos iteration (\cite{Alikakos}),we have that
\begin{equation}\label{n_Linf_Linf_in}
\sup_{0\leq s\leq T_{\star\star}}||n(s)||_\infty\leq C_{n,\infty}^{in}(n_{in}).
\end{equation}

\textbf{Second step:} We prove the improvement to \eqref{Hyp_on_paxn}. First we estimate the time evolution of $||\pa_{xx}c_{\neq}||_{2p}^{2p}$:
\begin{align*}
\frac{1}{2p}\frac{d}{dt}||\pa_{xx}c_{\neq}||_{2p}^{2p}\leq&-\frac{2p-1}{2Ap^2}||\na(\pa_{xx}c_{\neq})^p||_2^2+\frac{p}{A}||\pa_x n||_{2p}^2||\pa_{xx}c_{\neq}||_{2p}^{2p-2}-\frac{1}{A}||\pa_{xx}c_{\neq}||_{2p}^{2p}.
\end{align*}
Here we use the fact that $\pa_x n=\pa_x n_{\neq}$. As a result, we see that
\begin{equation}\label{paxxcneq2p}
||\pa_{xx}c_{\neq}(t)||_{2p}\lesssim\sqrt{\frac{2pt}{A}}\sup_{0\leq s\leq t}||\pa_x n(s)||_{2p}+||\pa_x^2(c_{in})_{\neq}||_{2p},\quad \forall t\in [0,T_{\star\star}).
\end{equation}
By a similar argument as in the estimate of the term $||\pa_y c_{\neq}||_{2p}$ in \eqref{pa_y_c_neq_2p}, we have that
\begin{equation}\label{paxycneq2p}
|| \pa_{xy} c_{\neq}(t)||_{2p}\lesssim1+\frac{\sqrt{p}t^{3/2}}{A^{1/2}}\sup_{0\leq s\leq t}||\pa_x n(s)||_{2p}+p\frac{\sup_{0\leq s\leq t}||\pa_x n(s)||_{2p}^2}{A^{1/2}},\quad t\in [0,T_{\star\star}).
\end{equation}
Now we can calculate the time evolution of $||\pa_x n||_{2p}^{2p}$:
\begin{align}
\frac{1}{2p}\frac{d}{dt}||\pa_xn||_{2p}^{2p}\leq&-\frac{2p-1}{Ap^2}||\na(\pa_xn)^p||_2^2+\frac{2p-1}{Ap}\bigg(||\na (\pa_x n)^p||_2||\pa_x n||_{2p}^{p-1}||\pa_x\na c_{\neq}||_{2p}||n||_\infty\nonumber\\
+&||\na(\pa_x n)^p||_2||(\pa_x n)^p\na c||_{2}\bigg)\nonumber\\
=:&-\frac{2p-1}{Ap^2}||\na(\pa_x n)^p||_2^2+T_1+T_2.\label{nx infty two terms}
\end{align}
In the first line, we have used the fact that $\pa_x\na c=\pa_x\na c_{\neq}$. Now we need to separate the estimate into two cases, $p=1$ and $p\neq 1$. First we discuss the $p=1$ case. The $T_1$ term in \eqref{nx infty two terms} can be estimated using the $||\na \pa_x c_{\neq}||_{2p}$ estimates \eqref{paxxcneq2p} and \eqref{paxycneq2p} as follows:
\begin{align}
T_1\lesssim &\frac{1}{4A}||\na(\pa_xn)||_2^2+\frac{1}{A}||\pa_x\na c_{\neq}||_2^2||n||_\infty^2\nonumber\\
\lesssim&\frac{1}{4A}||\na(\pa_x n)||_2^2+\frac{1}{A}\bigg(\frac{2}{A^{2/3-\ep}}\sup_{0\leq s\leq t}||\pa_x n(s)||_2^2+1+A^{3\ep}\sup_{0\leq s\leq t}||\pa_x n(s)||_2^2+\frac{1}{A}\sup_{0\leq s\leq t}||\pa_x n(s)||_2^4\bigg)||n||_\infty^2\label{T_1_L2}
\end{align}
The $T_2$ in \eqref{nx infty two terms} can be estimated using $\na c_{\neq}$ $L^4$ estimates  \eqref{pax_cneq_L4_control}, \eqref{paycneq_L4_ctrl}, Lemma \ref{nasc0Lpestimate}, Gagliardo-Nirenberg-Sobolev inequality and H\"older inequality as follows:
\begin{align}
T_2\lesssim&\frac{1}{4A}||\na(\pa_x n)||_2^2+\frac{1}{A}||\pa_xn||_2^2||\na c||_4^4\nonumber\\
\lesssim&\frac{1}{4A}||\na(\pa_x n)||_2^2+\frac{B}{A}||\pa_xn||_2^2\bigg(1+M^4+||\pay (c_{in})_0||_4^4+\frac{t^2}{A^2}\sup_{0\leq s\leq t}||n(s)||_4^4\nonumber\\
&+\frac{t^6}{A^2}\sup_{0\leq s\leq t}||n(s)||_4^4+\frac{1}{A^2}\sup_{0\leq s\leq t}||n(s)||_4^8\bigg).\label{T_2_L2}
\end{align}
Now combining \eqref{n_Linf_Linf_in}, \eqref{nx infty two terms}, \eqref{T_1_L2}, \eqref{T_2_L2}, Lemma \ref{Lem:nasc0Lpestimate},  and, we obtain that
\begin{equation*}
\frac{d}{dt}||\pa_x n||_2^2\leq \frac{C(\CC,\pa_y(c_{in})_0)}{A^{1-6\ep}}(1+\sup_{0\leq s\leq t}||\pa_x n(s)||_2^4).
\end{equation*}
Now use a comparison argument similar to the one used to prove \eqref{n_4_bound_initial_layer}, we end up with the following estimate given $A$ chosen large enough
\begin{equation}\label{pa_x_n_2_initial_layer}
||\pa_x n(t)||_2\leq C_{\pa_xn, L^2}(n_{in}), \quad \forall t\in [0,T_{\star\star}).
\end{equation}
This finishes the treatment of the case $p=1$.

For the $p\neq 1$ case, there exists a large $B$ such that the $T_1$ term in \eqref{nx infty two terms} can be estimated as follows:
\begin{align*}
T_1\leq\frac{2p-1}{BAp^2}||\na(\pa_x n)^p||_2^2+\frac{BCp^3}{A}||(\pa_x n)^p||_1^{2(1-\frac{2}{p+1})}||\na\pa_x c_{\neq}||_{2p}^{\frac{4p}{p+1}}||n||_\infty^{\frac{4p}{p+1}},
\end{align*}
which combined with $\na \pa_x c_{\neq}$ $L^{2p}$ estimates \eqref{paxxcneq2p}, \eqref{paxycneq2p}, hypothesis \eqref{Hyp_on_paxn} and $L^2$ estimate of $\pa_x n$ in the initial time layer \eqref{pa_x_n_2_initial_layer} yields
\begin{align}\label{T_1}
T_1\leq \frac{2p-1}{BAp^2}||\na(\pa_xn)^p||_2^2+\frac{Bp^7}{A^{1-6\ep}}||(\pa_x n)^p||_1^{2(1-\frac{2}{p+1})}C(C_{2,\infty}, C_{\pa_x n,\infty}, n_{in}).
\end{align}
For the $T_2$ in \eqref{nx infty two terms}, we can estimate it using Lemma \eqref{nasc0Lpestimate}, $L^\infty$ estimate of $n$ \eqref{n_Linf_Linf_in} and $\na c_{\neq}$ $L^{2p}$ estimates \eqref{pa_x_c_neq_2p}, \eqref{pa_y_c_neq_2p} as follows:
\begin{align}
T_2\leq&\frac{2p-1}{Ap}||\na (\pa_x n)^p||_2||(\pa_x n)^p||_{16/7}||\na c||_{16}\nonumber\\
\leq&\frac{2p-1}{BAp^2}||\na(\pa_x n)^p||_2^2+\frac{BCp^4}{A}||(\pa_x n)^p||_1^2||\na c||_{16}^{32/7}\nonumber\\
\leq&\frac{2p-1}{BAp^2}||\na(\pa_x n)^p||_2^2+\frac{Bp^4}{A^{1-7\ep}}||\pa_x n||_p^{2p}C(\CC).\label{T_2}
\end{align}
Combining \eqref{nx infty two terms}, \eqref{T_1} and \eqref{T_2} and integrating in time, we have that
\begin{align}
\frac{1}{2p}||\pa_xn(t)||_{2p}^{2p}\leq&
\frac{1}{2p}||\pa_xn_{in}||_{2p}^{2p}+ \frac{Bp^7t}{A^{1-6\ep}}\sup_{0\leq s\leq t}||\pa_x n(s)||_p^{2p(1-\frac{2}{p+1})}C(C_{2,\infty}, C_{\pa_x n,\infty}, n_{in})\nonumber\\
&+\frac{Bp^4t}{A^{1-7\ep}}\sup_{0\leq s\leq t}||\pa_x n(s)||_p^{2p}C(C_{2,\infty}),\quad \forall t\in [0, T_{\star\star}].\label{pax_n_2p_initial_layer}
\end{align}

Finally, we use the \eqref{pax_n_2p_initial_layer} together with \eqref{pa_x_n_2_initial_layer} to get the $||\pa_xn||_{L^\infty_t(0, T_{\star\star};L_{x,y}^\infty)}\leq 2 C_{\pa_xn,\infty}$. Note that if for $\forall j\in \mathbb{N}$, $\sup_{0\leq s<T_{\star\star}}||\pa_xn(s)||_{2^j}\leq 1$, we have that $\sup_{0\leq s<T_{\star\star}}||\pa_xn(s)||_\infty\leq 1$, and the result follows. Therefore, we assume that there exists $4\leq p_\star=2^{j_\star}\in 2^\mathbb{N}$ such that it is the first integer that $\sup_{0\leq s<T_{\star\star}}||\pa_xn||_{p_\star}\geq 1$. For $p=p_\star/2$,
\begin{equation}\label{paxnLpstar/2}
||\pa_x n||_{L_t^\infty(0,T_{\star\star};L^{p_\star/2}_{x,y})}\leq \max\{C_{\pa_x n,L^2},1\}.
\end{equation}
We will only care about $p=2^j>p_\star$, $j\in \mathbb{N}$. By the H\"{o}lder's inequality,
\begin{equation*}
1\leq \sup_{0\leq s\leq T_{\star\star}}||\pa_xn(s)||_{p_*}\leq \sup_{0\leq s\leq T_{\star\star}}||\pa_xn(s)||_2^{\theta}\sup_{0\leq s\leq T_{\star\star}}||\pa_xn(s)||_{p}^{1-\theta}, \quad p> p_\star, p\in 2^{\mathbb{N}}.
\end{equation*}
Now combining this with \eqref{pa_x_n_2_initial_layer}, we have a lower bound for $\sup_{0\leq s\leq T_{\star\star}}||\pa_x n(s)||_{p}$:
\begin{equation}
\sup_{0\leq s\leq T_{\star\star}}||\pa_xn(s)||_{p}\geq (1+C_{\pa_x n,L^2})^{-3}, \quad\forall p\geq p_\star,\enskip p\in 2^{\mathbb{N}}.
\end{equation}
Now combining this with \eqref{pax_n_2p_initial_layer}, we have that
\begin{align}\sup_{0\leq s\leq T_{\star\star}}||\pa_xn(s)||_{2p}^{2p}\leq&
||\pa_xn_{in}||_{2p}^{2p}+\frac{Bp^8}{A^{2/3-8\ep}}\sup_{0\leq s\leq T_{\star\star}}||\pa_x n(s)||_p^{2p}C(C_{2,\infty}, C_{\pa_x n,\infty}, n_{in}).
\end{align}
Now we can take the $A$ large such that
\begin{align}\sup_{0\leq s\leq T_{\star\star}}||\pa_xn(s)||_{2p}^{2p}\leq&
||\pa_xn_{in}||_{2p}^{2p}+p^8\sup_{0\leq s\leq T_{\star\star}}||\pa_x n(s)||_p^{2p}.
\end{align}
Combining $L^2$ estimate of $\pa_x n$ \eqref{pa_x_n_2_initial_layer}, $L^{p_\star/2}$ estimate of $\pa_x n$ \eqref{paxnLpstar/2} and the standard Moser-Alikakos iteration yields
\begin{equation}\label{paxn_Linf_in_con}
\sup_{0\leq s\leq T_{\star\star}}||\pa_x n(s)||_{\infty}\leq C(n_{in}).
\end{equation}
Now by picking $2C_{\pa_xn,\infty}\geq C(n_{in})$, we finishes the proof of the improvement to \eqref{Hyp_on_paxn}.

\textbf{Third step:} We prove the \eqref{ctrl:na_c_Linf_in}. First we calculate the time evolution of $||\pa_x c_{\neq}||_{2p}$ using \eqref{pa_x_n_2_initial_layer} and \eqref{paxn_Linf_in_con}:
\begin{align*}
\frac{1}{2p}\frac{d}{dt}||\pa_x c_{\neq}||_{2p}^{2p}\leq \frac{1}{A}||\pa_x n||_{2p}||\pa_x c_{\neq}||_{2p}^{2p-1}\leq \frac{1}{A}(C_{\pa_xn,2}(n_{in})+C_{\pa_xn,\infty}(n_{in}))||\pa_x c_{\neq}||_{2p}^{2p-1}.
\end{align*}
This implies that
\begin{equation}\label{paxcneq_L2p_in}
||\pa_x c_{\neq}(t)||_{2p}\leq \frac{t}{A}(C_{\pa_xn,2}+C_{\pa_xn,\infty})+||\pa_x(c_{in})_{\neq}||_{2p}, \quad\forall p\in[2,\infty).
\end{equation}
Therefore, by the assumption that $||\na (c_{in})_{\neq}||_{H^1\cap W^{1,\infty}}\leq CA^{-q},\enskip q>1/2$, we have that
\begin{equation}
||\pa_x c_{\neq}(t)||_\infty\leq \frac{t}{A}(C_{\pa_xn,2}+C_{\pa_xn,\infty})+CA^{-q}.
\end{equation}
For $t\leq T_{\star\star}\leq A^{1/3+\ep}$, we have the following estimate for $A$ chosen large enough
\begin{equation}\label{paxcneq_Linf_in}
||\pa_x c_{\neq}(t)||_\infty\leq 1,\quad \forall t\in[0,T_{\star\star}).
\end{equation}

In order to estimate the norm $||\pa_y c_{\neq}||_{2p}$, we need to introduce a time weighted  norm. To define it, we first consider the following simpler equation only taking into account the strong shear flow destabilizing effect
\begin{equation*}
\frac{d}{dt}f=-u'(y)\pa_x c_{\neq}-u(y)\pa_x f,\quad f_{in}=\pay (c_{in})_{\neq}.
\end{equation*}
We can estimate the time evolution of the $L^{2p}$ norm of the solution using \eqref{paxcneq_L2p_in} as follows:
\begin{equation}
\frac{1}{2p}\frac{d}{dt}||f||_{2p}^{2p}\leq ||u'\pa_x c_{\neq}||_{2p}||f||_{2p}^{2p-1}\leq \left(\frac{t}{A}(C_{\pa_xn,2}+C_{\pa_xn,\infty})+CA^{-q}\right)||u'||_\infty||f||_{2p}^{2p-1}.
\end{equation}
Time integration yields
\begin{equation}
||f(t)||_{2p}\leq \frac{t^2}{A}||u'||_\infty(C_{\pa_xn,2}+C_{\pa_xn,\infty})+C||u'||_\infty A^{-q}t+CA^{-q}=:G_\infty(t), 0\leq t<T_{\star\star},\forall p\geq 2.
\end{equation}
Note the following relation:
\begin{equation}
G_\infty'(t)=2\frac{t}{A}(C_{\pa_xn,2}+C_{\pa_xn,\infty})||u'||_\infty+C||u'||_\infty A^{-q}\geq ||u'\pa_x c_{\neq}||_{2p}.
\end{equation}
Next we consider the following time weighted norm:
\begin{equation}
\mathcal{F}_{p}^{1/p}(t):=\frac{||\pa_y c_{\neq}(t)||_{p}}{e^{G_\infty(t)}}.
\end{equation}
Since $G_\infty$ is bounded by a universal constant if we choose $A$ large enough, the norm $\mathcal{F}_{p}^{1/p}$ is equivalent to the $L^p$ norm. However, the quantity $\mathcal{F}_{p}$ has better property than the usual $L^{p}$ norm. When we take the time derivative of $\mathcal{F}_{p}$, the weight $\frac{1}{e^{pG_\infty(t)}}$ will contribute extra negative term to compensate for the strong shear flow destabilizing effect.  

The time derivative of the $\mathcal{F}_{2p}$ can be estimated with the $L^\infty$ bound of $n$ in the initial time layer \eqref{n_Linf_Linf_in} and Gagliardo-Nirenberg-Sobolev inequality as follows
\begin{align*}
\frac{d}{dt}\mathcal{F}_{2p}\leq\frac{2p}{e^{G_\infty 2p}}\bigg(&-\frac{2p-1}{Cp^2A}\frac{||\pa_y c_{\neq}||_{2p}^{4p}}{||\pa_y c_{\neq}||_{p}^{2p}}+\frac{2p-1}{A}||\pa_y c_{\neq}||_{2p}^{2p-2} (M+C^{in}_{n,\infty})^2\\
&+||\pa_y c_{\neq}||_{2p}^{2p-1}||u'\pa_x c_{\neq}||_{2p}-G_\infty'(t)||\pa_y c_{\neq}||_{2p}^{2p}\bigg).
\end{align*}
If $\sup_{0\leq s\leq T_{\star\star}}\frac{||\pa_y c_{\neq}(s)||_{2p}}{e^{G_\infty(s)}}\leq 1$, we have
\begin{equation}\label{case_1}
\mathcal{F}_{2p}(t)\leq 1,\quad \forall t\in [0, T_{\star\star}].
\end{equation}
Otherwise if $\sup_{0\leq s\leq  T_{\star\star}}\frac{||\pa_y c_{\neq}(s)||_{2p}}{e^{G_\infty(s)}} \geq 1$, we have that at the maximum point $t_\star$ of $\mathcal{F}_{2p}$, $||\pa_y c_{\neq}(t_\star)||_{2p}\geq 1$ and the following holds:
\begin{align*}
\frac{d}{dt}\mathcal{F}_{2p}\bigg|_{t=t_\star}\leq\frac{2p}{e^{G_\infty (t_\star)2p}}\bigg(-\frac{2p-1}{CAp^2}\frac{||\pa_y c_{\neq}(t_\star)||_{2p}^{2p}/e^{G_{\infty}(t_\star) 2p}}{(||\pa_y c_{\neq}(t_\star)||_{p}^{p}/e^{G_\infty(t_\star) p})^2}+\frac{2p-1}{A}(M+C^{in}_{n,\infty})^2\bigg)||(\pa_y c_{\neq})^p(t_\star)||_2^2.
\end{align*}
Now we have that
\begin{equation}\label{case_2}
\sup_{0\leq s\leq T_{\star\star}}F_{2p}(s)\leq C(M, C_{n,\infty}^{in})p^2 \sup_{0\leq s\leq T_{\star\star}}F_p(s)^2+\frac{||\pay (c_{in})_{\neq}||_{2p}^{2p}}{e^{2pCA^{-q}}}.
\end{equation}
Combining \eqref{case_1} and \eqref{case_2}, and noting that $C_{n,\infty}^{in}$ only depends on $n_{in}$ \eqref{n_Linf_Linf_in}, we have
\begin{equation}
\sup_{0\leq s\leq T_{\star\star}} F_{2p}(s)\leq \max\{C(M, n_{in})p^2\sup_{0\leq s\leq T_{\star\star}} F_p^2(s),1\}
\end{equation}
for $A$ large enough. Combining this with the fact that $||\pa_y c_{\neq}||_2\leq \sqrt{C_{ED}}(||n_{in}||_{H^1}+1)<\infty$ from the hypothesis \eqref{H2} and using similar Moser-Alikakos iteration argument as before, we end up with
\begin{equation}
\sup_{0\leq s\leq  T_{\star\star}}||\pa_y c_{\neq}(s)||_\infty\leq C(M, C_{ED},n_{in}).
\end{equation}
Combining this with \eqref{paxcneq_Linf_in}, we have proven that
\begin{equation}
||\na c_{\neq}(t)||_\infty\leq C^{in}_{\na c_{\neq},\infty}(M, C_{ED},n_{in}), \quad\forall t\in [0,T_{\star\star}).
\end{equation}
Now since we have proven the bootstrap conclusion \eqref{paxn_Linf_in_con}, $T_{\star\star}$ can be extended all the way to $A^{1/3+\ep},\ep<\frac{1}{12}$. Therefore all the estimates we got above can be extended to $[0, A^{1/3+\ep}]$. This completes the proof of the lemma.
\end{proof}
\subsection{Long time estimate}
In this subsection, we prove \eqref{ctrl:n_Linf} and \eqref{ctrl:na_c_Linf} in the time interval $[A^{1/3+\ep}, T_\star)$. The battle plan is as follows:
\begin{align*}
\text{Hypothesis \eqref{H2}}\overbrace{\Rightarrow}^{\text{$L^4$ energy estimate}}||\na c_{\neq}(t)||_4, ||n_{\neq}(t)||_4 \text{ estimates}\overbrace{\Rightarrow}^{\text{Moser-Alikakos iteration }}||\na c_{\neq}||_\infty,||n_{\neq}||_\infty\text{ estimates}.
\end{align*}

\textbf{First step:} We estimate the $||\na c_{\neq}||_4$. First we need to get an estimate of the norm at the starting time $t_0:=A^{1/3+\ep}$. By standard energy estimate combining with \eqref{H2}, \eqref{paxcneq_L2p_in}, \eqref{ctrl:n_Linf_in} and \eqref{ctrl:na_c_Linf_in}, we have that
\begin{eqnarray}\label{A1/3L4data}
||\na c_{\neq}(t_0)||_4^4\leq 2,
\end{eqnarray}
if $A$ is large enough.

For $t\geq A^{1/3+\ep}$, applying the estimate \eqref{dt_pax_c_L4} and the bootstrap hypothesis \eqref{H2},\eqref{H4} and \eqref{H5} and the H\"{o}lder's inequality, we have that
\begin{align*}
\frac{d}{dt}||\pa_x c_{\neq}||_4^4\lesssim \frac{1}{A}||\pa_x c_{\neq}||_4^2||n_{\neq}||_4^2\lesssim \frac{1}{A}||\pa_x c_{\neq}||_2||n_{\neq}||_2||\pa_x c_{\neq}||_\infty||n_{\neq}||_\infty\lesssim \frac{1}{A}e^{-\eta\frac{t}{A^{1/3}}}C_{ED}(||n_{in}||_{H^1}^2+1)C_{2,\infty}^2.
\end{align*}
Time integrating the above inequality and combining it with \eqref{A1/3L4data}, we obtain the following estimate by taking $A$ large:
\begin{equation}\label{pax_c_L4_long_time}
||\pa_x c_{\neq}(t)||_4^4\leq 4,\quad \forall t\in[A^{1/3+\ep}, T_\star).
\end{equation}
Applying the time evolution estimate of $||\pay c_{\neq}||_4^4$ \eqref{pa_y_c_neq_4}, the fact that $||u'||_\infty\leq C$, the bootstrap hypothesis \eqref{H2},\eqref{H4} and \eqref{H5} and the H\"{o}lder's inequality, we can estimate the time evolution of $||\pa_y c_{\neq}||_4^4$ as follows:
\begin{align*}
\frac{d}{dt}||\pa_y c_{\neq}||_4^4\lesssim &\frac{1}{A}||\pa_y c_{\neq}||_4^2||n_{\neq}||_4^2+\frac{1}{A}||\pa_y c_{\neq}||_4^3(||u'\pa_x c_{\neq}||_4A)\\
\lesssim &\frac{1}{A}||\pa_y c_{\neq}||_2||\pa_y c_{\neq}||_\infty||n_{\neq}||_2||n_{\neq}||_\infty+\frac{1}{A}||\pa_y c_{\neq}||_2^{3/2}||\pa_yc_{\neq}||_\infty^{3/2}(||\pa_x c_{\neq}||_\infty^{1/2}||\pa_x c_{\neq}||_2^{1/2}A)\\
\lesssim& \frac{1}{A}e^{-\eta\frac{t}{A^{1/3}}}C_{ED}(||n_{in}||_{H^1}^2+1)C_{2,\infty}^2+\frac{1}{A}C_{ED}(||n_{in}||_{H^1}^2+1)e^{-\eta\frac{3t}{4A^{1/3}}}C_{2,\infty}^{2}\left(e^{-\frac{\eta}{4}A^\ep}A\right).
\end{align*}
Note that $e^{-\frac{\eta A^{\ep}}{4}}A\leq C(\ep,\eta)$. Now integrate this in time and use the initial condition \eqref{A1/3L4data}, we have the following estimate by picking $A$ large:
\begin{equation}\label{pay_c_L4_long_time}
||\pa_y c_{\neq}(t)||_4^4\leq 4,\quad \forall t\in[A^{1/3+\ep}, T_\star).
\end{equation}
This concludes the first step.


\textbf{Second Step:} We estimate the time evolution of $||n||_4^4$:
\begin{align*}
\frac{d}{dt}\int|n|^4dxdy=&4\int n^3\bigg(\frac{\de n-\na\cdot(\na c n)}{A}\bigg)dxdy\\
\lesssim&-\int\frac{|\na(n)^2|^2}{A}dxdy+\frac{||n^2||_4^2||\na c||_4^2}{A}.
\end{align*}
Applying the following Gagliardo-Nirenberg-Sobolev inequalities
\begin{align*}
||n^2||_2\lesssim||\na(n)^2||_2^{1/2}||n^2||_1^{1/2},\\
||n^2||_4\lesssim||\na(n^2)||_2^{1/2}||n^2||_2^{1/2}
\end{align*}
and Young's inequality in the above differential inequality yields 
\begin{align*}
\frac{d}{dt}\int|n|^4dxdy
\lesssim-\frac{||n^2||_2^4}{CA||n^2||_1^{2}}+\frac{||n^2||_2^2||\na c||_4^4}{A}
\end{align*}
Combining this with the $L^4$ estimates of $\na c_{\neq}$ \eqref{pax_c_L4_long_time}, \eqref{pay_c_L4_long_time}, initial time estimate \eqref{ctrl_in} Lemma \ref{nasc0Lpestimate} and hypothesis \eqref{H2}, we have
\begin{align}\label{n_L4_long_time}
\sup_{t_0\leq t\leq T_\star}||n(t)||_4^4\lesssim ||n(t_0)||_4^4+\sup_{t_0\leq t\leq T_\star}||n^2(t)||_1^2\sup_{t_0\leq t\leq T_\star}||\na c(t)||_4^{4}\leq \left(C_{n,L^4}^{long}(n_{in},\pay( c_{in})_0,C_{n_0,L^2},C_{ED},M)\right)^4.
\end{align}
Since $||n_0||_{L^4(\Torus\times\rr)}\leq ||n||_{L^4(\Torus\times\rr)}$, we have $||n_{\neq}||_4\leq 2||n||_4\leq 2C^{long}_{n,L^4}$.

\textbf{Third Step:} Now we can start to do the Moser-Alikakos iteration on $||\pa_y c_{\neq}||_{2p},\enskip p\in 2^{\mathbb{N}}$ to get $||\pa_y c_{\neq}||_\infty$ bound on $[A^{1/3+\ep},T_\star]$. The time evolution of $||\pay c_{\neq}||_{2p}^{2p}$ can be estimated as follows
\begin{align}
\frac{d}{dt}\int|\pa_y c_{\neq}|^{2p}dxdy=&2p\int(\pa_yc_{\neq})^{2p-1}\left(\frac{\de \pa_y c_{\neq}+\pa_y n_{\neq}}{A}-u'(y)\pa_x c_{\neq}\right)dxdy\nonumber\\
\leq&-\left(2-\frac{2}{p}\right)\frac{1}{A}\int|\na(\pa_yc_{\neq})^p|^2dxdy+\frac{2p^2}{A}||\pa_y c_{\neq}||_{4p-4}^{2p-2}||n_{\neq}||_4^2+\frac{2p}{A}||\pa_y c_{\neq}||_{4p-2}^{2p-1}||u'\pa_x c_{\neq}||_2A.\label{pay_c_neq_L2p_longtime}
\end{align}
Note that from hypothesis \eqref{H2}, for $A$ chosen large enough, we have that
\begin{align}\label{drift control 2}
||u'\pa_x c_{\neq}||_2A\leq \sqrt{C_{ED}}(||n_{in}||_{H^1}+1)e^{-\eta t/(2A^{1/3})}||u'||_\infty A\leq C(\eta,\epsilon), \quad t\geq A^{1/3+\ep}.
\end{align}
We also have the following type of H\"{o}lder's inequality:
\begin{align*}
||\pa_y c_{\neq}||_{4p-4}\leq||\pa_y c_{\neq}||_{4p}^{\frac{1-\frac{1}{2p-2}}{1-{1}/(2p)}}||\pa_y c_{\neq}||_2^{1-\frac{1-\frac{1}{2p-2}}{1-1/(2p)}};\\
||\pa_y c_{\neq}||_{4p-2}\leq||\pa_y c_{\neq}||_{4p}^{\frac{1-\frac{1}{2p-1}}{1-{1}/(2p)}}||\pa_y c_{\neq}||_2^{1-\frac{1-\frac{1}{2p-1}}{1-1/(2p)}}.\\
\end{align*}
Applying all these estimates together with Gagliardo-Nirenbery-Sobolev inequality and hypothesis \eqref{H2} in the above differential inequality \eqref{pay_c_neq_L2p_longtime}, we have that
\begin{align*}
\frac{d}{dt}&\int|\pa_y c_{\neq}|^{2p}dxdy\\
\lesssim& -\frac{||\na(\pa_y c_{\neq})^p||_2^2}{A}+\frac{p^2}{A}||\pa_y c_{\neq}||_{4p}^{\frac{2p-3}{1-1/(2p)}}||\pa_y c_{\neq}||_2^{2p-2-\frac{2p-3}{1-1/(2p)}}||n_{\neq}||_4^2+\frac{p}{A}||\pa_y c_{\neq}||_{4p}^{\frac{2p-2}{1-1/(2p)}}||\pa_y c_{\neq}||_2^{2p-1-\frac{2p-2}{1-1/(2p)}}C\\
\lesssim& -\frac{||\na(\pa_y c_{\neq})^p||_2^2}{A}+\frac{p^2}{A}(||\na(\pa_y c_{\neq})^p||_{2}||(\pa_y c_{\neq})^p||_2)^{\frac{2p-3}{2p-1}}||\pa_y c_{\neq}||_2^{2p-2-\frac{2p-3}{1-1/(2p)}}||n_{\neq}||_4^2\\
&+\frac{p}{A}(||\na(\pa_y c_{\neq})^p||_2||(\pa_y c_{\neq})^p||_2)^{\frac{2p-2}{2p-1}}||\pa_y c_{\neq}||_2^{2p-1-\frac{2p-2}{1-1/(2p)}}C\\
\lesssim&-\frac{||\na(\pa_y c_{\neq})^p||_2^2}{2A}+ \frac{p^4||(\pa_y c_{\neq})^p||_2^{\frac{2(2p-3)}{2p+1}}}{A}\left(\sqrt{ C_{ED}}||n_{in}||_{H^1}+1\right)^2\left(C_{n,L^4}^{long}\right)^4\\
&+\frac{p^2||(\pa_y c_{\neq})^p||_2^{\frac{(2p-2)}{p}}}{A}\left(\sqrt{C_{ED}}||n_{in}||_{H^1}+1\right)^2C^2.
\end{align*}
Now use the following Nash inequality
\begin{align*}
||f||_2\lesssim||f||_1^{1/2}||\na f||_2^{1/2}
\end{align*}
we have that
\begin{align*}
\frac{d}{dt}\int|\pa_y c_{\neq}|^{2p}dxdy\lesssim-\frac{||(\pa_y c_{\neq})^p||_2^{4}}{AC||(\pa_y c_{\neq})^p||_1^{2}}+\frac{p^4}{A}\bigg(||(\pa_y c_{\neq})^p||_2^{\frac{2(2p-3)}{2p+1}}+||(\pa_y c_{\neq})^p||_2^{\frac{2p-2}{p}}\bigg)C(C_{ED}, C_{n_0,L^2}, n_{in},\pay (c_{in})_0).
\end{align*}
This estimate leads to the following estimate
\begin{align*}
\sup_{t_0\leq t\leq T_\star}||(\pa_y c_{\neq})^p(t)||_2^2\leq \max\left\{p^4C(C_{ED},C_{n_0,L^2}, n_{in},\pay (c_{in})_0)\sup_{t_0\leq t\leq T_\star}||(\pa_y c_{\neq})^p||_1^2,1,||(\pa_y c_{\neq})^p(t_0)||_2^2\right\}.
\end{align*}
Now recalling hypothesis \eqref{H2} and the bound \eqref{ctrl:na_c_Linf_in}, we use the Moser-Alikakos iteration to get the bound
\begin{equation}\label{nay_c_neq_infty_long_time}
||\pa_y c_{\neq}||_\infty\leq C(n_{in},C_{ED}, C_{n_0,L^2},\pay (c_{in})_0) ,\quad\forall t\in [t_0,T_\star].
\end{equation}

The proof of $||\pa_x c_{\neq}||_{\infty}$ is similar but easier, so we omit the proof for the sake of brevity. As a result, we obtain
\begin{equation}
||\na c_{\neq}(t)||_\infty\leq  C_{\na c_{\neq},\infty}^{long}(n_{in},C_{ED}, C_{n_0,L^2},\pay (c_{in})_0) ,\quad \forall t\in [t_0,T_\star].
\end{equation}
Now applying the Moser-Alikakos iteration and \eqref{ctrl:n_Linf_in}, we have that
\begin{equation}
||n(t)||_\infty\leq C_{n,\infty}^{long}(C_{ED},n_{in}, C_{n_0,L^2},\pa_y c_{in}),\quad\forall t\in[t_0,T_\star].
\end{equation}
By picking the $C_{n,\infty}\gg \max\{C_{n,\infty}^{in},C_{n,\infty}^{long}\}$ in \eqref{H4}, we prove \eqref{ctrl:n_Linf}.
By picking the $C_{\na c_{\neq},\infty}\gg \max\{C_{\na c_{\neq},\infty}^{in},C_{\na c_{\neq},\infty}^{long}\}$ in \eqref{H5}, we prove \eqref{ctrl:na_c_Linf}.
This completes the proof of Proposition \ref{prop:boot2D}.

\appendix
\section{Appendix}

\begin{lem}\label{Lem:nasc0Lpestimate}Consider the solution to \eqref{zeromode} subject to initial data $(c_{in})_0$. For $\forall s\in \mathbb{N}$ and any $(p,q)$ pair such that either $2\leq p<\infty, 1\leq q\leq p$ or $p=\infty,1< q\leq p$ is satisfied, the following estimates hold for the solution $c_0$
\bel\label{nasc0Lpestimate}\ba
||\pay c_0(t)||_p\lesssim_{p,q} &\sup_{0\leq \tau\leq t}||n_0(\tau)||_q+||(\pay c_{in})_0||_p;\\
|| \pay^{s+1} c_0(t)||_p\lesssim_{p,q}&\sup_{0\leq \tau\leq t}||\pay^s n_0(\tau)||_q+||(\pay^{s+1}c_{in})_0||_p,\quad 2\leq p\leq \infty.
\ea\eel
\end{lem}

\begin{proof}
For $\forall 2\leq p\leq \infty$, using the heat mild solution representation, Minkowski's integration inequality and Young's inequality, we have the following
\begin{align*}
||\pa_y c_0||_p\leq&\norm{\int_0^t \int_\rr e^{\frac{s-t}{A}}\frac{-(x-y)}{2(t-s)A^{-1}}\frac{1}{\sqrt{4\pi (t-s)A^{-1}}}e^{\frac{-|x-y|^2}{4(t-s)A^{-1}}}\frac{n_0}{A}
(y,s)dyds}_p+||e^{-\frac{t}{A}}e^{t\frac{\de}{A}}(\pay c_{in})_0||_p\\
\lesssim&\int_0^t \frac{e^{\frac{s-t}{A}}}{2A^{-1}(t-s)}\norm{\int_\rr \frac{|x-y|}{\sqrt{4\pi A^{-1}(t-s)}}e^{\frac{-|x-y|^2}{4A^{-1}(t-s)}}\frac{n_0}{A}(y,s)dy}_pds+||(\pay c_{in})_0||_p\\
\lesssim&\int_0^t \frac{e^{\frac{s-t}{A}}}{2A^{-1}(t-s)}\norm{ \frac{|y|}{\sqrt{4\pi A^{-1}(t-s)}}e^{\frac{-|y|^2}{4A^{-1}(t-s)}}}_r\norm{\frac{n_0}{A}(s)}_qds+||(\pay c_{in})_0||_p\\
\lesssim&\int_0^t \frac{e^{\frac{s-t}{A}}}{2(\frac{t-s}{A})^{1-\frac{1}{2r}}}\frac{ds}{A}\sup_{0\leq s\leq t}||n_0(\cdot,s)||_q+||(\pay c_{in})_0||_p\\
\lesssim&\sup_{0\leq s\leq t}||n_0(s)||_q+||(\pay c_{in})_0||_p.
\end{align*}
Here $\frac{1}{p}+1=\frac{1}{q}+\frac{1}{r}$ and $1\leq r<\infty$. The proof for the higher derivative case is similar, so we omit the proof. This concludes the proof of the lemma.
\end{proof}

\vfill\eject
\bibliographystyle{abbrv}
\bibliography{nonlocal_eqns,JacobBib,SimingBib}

\def\cprime{$'$}
\begin{thebibliography}{10}

\bibitem{Alikakos}
N.~Alikakos.
\newblock {$L^p$} bounds of solutions to reaction-diffusion equations.
\newblock {\em Comm. Part. Diff. Eqn.}, 4:827--868, 1979.

\bibitem{BeckWayne11}
M.~Beck and C.~Wayne.
\newblock Metastability and rapid convergence to quasi-stationary bar states
  for the two-dimensional {Navier--Stokes} equations.
\newblock {\em Proc. Royal Soc. of Edinburgh: Sec. A Mathematics},
  143(05):905--927, 2013.

\bibitem{BedrossianIA10}
J.~Bedrossian.
\newblock Intermediate asymptotics for critical and supercritical aggregation
  equations and {Patlak-Keller-Segel} models.
\newblock {\em Comm. Math. Sci.}, 9:1143--1161, 2011.

\bibitem{BCZ15}
J.~Bedrossian and M.~Coti~Zelati.
\newblock Enhanced dissipation, hypoellipticity, and anomalous small noise
  inviscid limits in shear flows.
\newblock {\em Archive for Rational Mechanics and Analysis, Vol:224,
  ISSN:0003-9527}, pages 1161--1204, 2017.

\bibitem{BGM15III}
J.~Bedrossian, P.~Germain, and N.~Masmoudi.
\newblock On the stability threshold for the {3D Couette} flow in {Sobolev}
  regularity.
\newblock {\em Annals of Mathematics 185 (2017)}, pages 541--608.

\bibitem{BGM15I}
J.~Bedrossian, P.~Germain, and N.~Masmoudi.
\newblock Dynamics near the subcritical transition of the {3D Couette flow I:
  Below} threshold.
\newblock {\em Mem. of the AMS., arXiv:1506.03720}, 2015.

\bibitem{BGM15II}
J.~Bedrossian, P.~Germain, and N.~Masmoudi.
\newblock Dynamics near the subcritical transition of the {3D Couette flow II:
  Above} threshold.
\newblock {\em arXiv:1506.03721}, 2015.

\bibitem{BedrossianHe16}
J.~Bedrossian and S.~He.
\newblock Suppression of blow-up in {Patlak-Keller-Segel} via shear flows.
\newblock {\em SIAM Journal on Mathematical Analysis, arXiv:1609.02866}, 2016.

\bibitem{BedrossianKim10}
J.~Bedrossian and I.~Kim.
\newblock Global existence and finite time blow-up for critical
  {Patlak-Keller-Segel} models with inhomogeneous diffusion.
\newblock {\em SIAM J. of Math. Anal.}, 45(3):934--964, 2013.

\bibitem{BMKS13}
J.~Bedrossian and N.~Masmoudi.
\newblock Existence, uniqueness and {Lipschitz} dependence for
  {Patlak-Keller-Segel} and {Navier-Stokes} in {$\mathbb{R}^2$} with
  measure-valued initial data.
\newblock {\em Arch. Rat. Mech. Anal.}, 214(3):717--801, 2014.

\bibitem{BMV14}
J.~Bedrossian, N.~Masmoudi, and V.~Vicol.
\newblock Enhanced dissipation and inviscid damping in the inviscid limit of
  the {Navier-Stokes} equations near the {2D Couette} flow.
\newblock {\em Arch. Rat. Mech. Anal.}, 216(3):1087--1159, 2016.

\bibitem{BRB10}
J.~Bedrossian, N.~Rodr\'iguez, and A.~Bertozzi.
\newblock Local and global well-posedness for aggregation equations and
  {Patlak-Keller-Segel} models with degenerate diffusion.
\newblock {\em Nonlinearity}, 24(6):1683--1714, 2011.

\bibitem{Biler95}
P.~Biler.
\newblock The {Cauchy} problem and self-similar solutions for a nonlinear
  parabolic equation.
\newblock {\em Studia Math.}, 114(2):181--192, 1995.

\bibitem{BilerCorriasDolbeault11}
P.~Biler, L.~Corrias, and J.~Dolbeault.
\newblock Large mass self-similar solutions of the parabolic-parabolic
  {Keller-Segel} model of chemotaxis.
\newblock {\em J. Math. Biol.}, 61(1):1--32, 2011.

\bibitem{BlanchetCalvezCarrillo08}
A.~Blanchet, V.~Calvez, and J.~Carrillo.
\newblock Convergence of the mass-transport steepest descent scheme for
  subcritical {Patlak-Keller-Segel} model.
\newblock {\em SIAM J. Num. Anal.}, 46:691--721, 2008.

\bibitem{BlanchetCarrilloMasmoudi08}
A.~Blanchet, J.~Carrillo, and N.~Masmoudi.
\newblock Infinite time aggregation for the critical {Patlak-Keller-Segel}
  model in $\mathbb{R}^2$.
\newblock {\em Comm. Pure Appl. Math.}, 61:1449--1481, 2008.

\bibitem{BlanchetEJDE06}
A.~Blanchet, J.~Dolbeault, and B.~Perthame.
\newblock Two-dimensional {Keller-Segel} model: Optimal critical mass and
  qualitative properties of the solutions.
\newblock {\em E. J. Diff. Eqn}, 2006(44):1--32, 2006.

\bibitem{CalvezCorrias}
V.~Calvez and L.~Corrias.
\newblock The parabolic-parabolic {K}eller-{S}egel model in {$\mathbb R^2$}.
\newblock {\em Commun. Math. Sci.}, 6(2):417--447, 2008.

\bibitem{CalvezCorriasEbde12}
V.~Calvez, L.~Corrias, and M.~A. Ebde.
\newblock Blow-up, concentration phenomenon and global existence for the
  {Keller-Segel} model in high dimension.
\newblock {\em Communications in Partial Differential Equations Vol. 37 , Iss.
  4}, pages 561--584, 2012.

\bibitem{CarrapatosoMischler17}
K.~Carrapatoso and S.~Mischler.
\newblock Uniqueness and long time asymptotic for the parabolic-parabolic
  {Keller-Segel} equation.
\newblock {\em Comm. Partial Differential Equations 42 (2017), no. 2}, pages
  291--345.

\bibitem{ChaeKangLee14}
M.~Chae, K.~Kang, and J.~Lee.
\newblock Global existence and temporal decay in {Keller-Segel} models coupled
  to fluid equations.
\newblock {\em Communications in Partial Differential Equations Vol. 39 , Iss.
  7}, pages 1205--1235, 2014.

\bibitem{CKRZ08}
P.~Constantin, A.~Kiselev, L.~Ryzhik, and A.~Zlato{\v{s}}.
\newblock Diffusion and mixing in fluid flow.
\newblock {\em Ann. of Math. (2)}, 168:643--674, 2008.

\bibitem{corrias2014existence}
L.~Corrias, M.~Escobedo, and J.~Matos.
\newblock Existence, uniqueness and asymptotic behavior of the solutions to the
  fully parabolic {Keller-Segel} system in the plane.
\newblock {\em Journal of Differential Equations Volume 257, Issue 6}, pages
  1840--1878, 2014.

\bibitem{CorriasPerthame06}
L.~Corrias and B.~Perthame.
\newblock Critical space for the parabolic-parabolic {Keller-Segel} model in
  {$\mathbb{R}^d$}.
\newblock {\em Comptes Rendus Mathematique Volume 342, Issue 10}, pages
  745--750, 2006.

\bibitem{CorriasPerthame08}
L.~Corrias and B.~Perthame.
\newblock Asymptotic decay for the solutions of the parabolic-parabolic
  {Keller-Segel} chemotaxis system in critical spaces.
\newblock {\em Mathematical and Computer Modelling, 47(7)}, pages 755--764,
  2008.

\bibitem{DuanLorzMarkowich10}
R.-J. Duan, A.~Lorz, and P.~Markowich.
\newblock Global solutions to the coupled chemotaxis-fluid equations.
\newblock {\em Comm. Partial Differential Equations, Vol. 35}, pages
  1635--1673, 2010.

\bibitem{FrancescoLorzMarkowich10}
M.~D. Francesco, A.~Lorz, and P.~Markowich.
\newblock Chemotaxis-fluid coupled model for swimming bacteria with nonlinear
  diffusion: global existence and asymptotic behavior.
\newblock {\em Discrete Contin. Dyn. Syst. Ser. A, Vol. 28}, pages 1437--1453,
  2010.

\bibitem{HillenPainter09}
T.~Hillen and K.~Painter.
\newblock A users guide to pde models for chemotaxis.
\newblock {\em Journal of mathematical biology, 58(1-2)}, pages 183--217, 2009.

\bibitem{Hortsmann}
D.~Horstmann.
\newblock {From 1970 until present: the Keller-Segel model in chemotaxis and
  its consequences}.
\newblock {\em I, Jahresber. Deutsch. Math.-Verein}, 105(3):103--165, 2003.

\bibitem{JagerLuckhaus92}
W.~J\"ager and S.~Luckhaus.
\newblock On explosions of solutions to a system of partial differential
  equations modelling chemotaxis.
\newblock {\em Trans. Amer. Math. Soc.}, 329(2):819--824, 1992.

\bibitem{KS}
E.~F. Keller and L.~Segel.
\newblock Model for chemotaxis.
\newblock {\em J. Theor. Biol.}, 30:225--234, 1971.

\bibitem{KiselevRyzhik12}
A.~Kiselev and L.~Ryzhik.
\newblock Biomixing by chemotaxis and enhancement of biological reactions.
\newblock {\em Communications in PDE, Vol. 37}, pages 298--318, 2012.

\bibitem{KiselevXu15}
A.~Kiselev and X.~Xu.
\newblock Suppression of chemotactic explosion by mixing.
\newblock {\em Archive for Rational Mechanics and Analysis, arXiv:1508.05333v3
  [math.AP]}, pages 1--36, 2015.

\bibitem{KozonoSugiyama09}
H.~Kozono and Y.~Sugiyama.
\newblock Global strong solution to the semi-linear {Keller-Segel} system of
  parabolic-parabolic type with small data in scale invariant spaces.
\newblock {\em Journal of Differential Equations Volume 247, Issue 1}, pages
  1--32, 2009.

\bibitem{LiuLorz11}
J.-G. Liu and A.~Lorz.
\newblock A coupled chemotaxis-fluid model: global existence.
\newblock {\em Annales de l’Institut Henri Poincaré. Analyse Non
  Linéaire, Vol. 28}, pages 643--652, 2011.

\bibitem{Lorz10}
A.~Lorz.
\newblock Coupled chemotaxis fluid model.
\newblock {\em Math. Models Methods Appl. Sci., Vol. 20}, pages 987--1004,
  2010.

\bibitem{Lorz12}
A.~Lorz.
\newblock A coupled {Keller-Segel-Stokes} model: global existence for small
  initial data and blow-up delay.
\newblock {\em Communications in Mathematical Sciences, Vol. 10}, pages
  555--574, 2012.

\bibitem{NagaiSenbaYoshida97}
T.~Nagai, T.~Senba, and K.~Yoshida.
\newblock Application of the {Trudinger-Moser} inequality to a parabolic system
  of chemotaxis.
\newblock {\em Funkcialaj Ekvacioj, 40}, pages 411--433, 1997.

\bibitem{Naito06}
Y.~Naito.
\newblock Asymptotically self-similar solutions for the parabolic system
  modelling chemotaxis.
\newblock {\em Self-similar solutions of nonlinear PDE, Banach Center
  Publications, Institute of mathematics, Polish academy of sciences, Warszawa,
  74}, pages 149--160.

\bibitem{Patlak}
C.~S. Patlak.
\newblock Random walk with persistence and external bias.
\newblock {\em Bull. Math. Biophys.}, 15:311--338, 1953.

\bibitem{Schweyer14}
R.~Schweyer.
\newblock Stable blow-up dynamic for the parabolic-parabolic
  {Patlak-Keller-Segel} model.
\newblock {\em arXiv:1403.4975 [math.AP]}, 2014.

\bibitem{TaoWinkler12}
Y.~Tao and M.~Winkler.
\newblock Boundedness in a quasilinear parabolic-parabolic {Keller-Segel}
  system with subcritical sensitivity.
\newblock {\em Journal of Differential Equations Volume 252, Issue 1}, pages
  692--715, 2012.

\bibitem{villani2009}
C.~Villani.
\newblock {\em Hypocoercivity}.
\newblock American Mathematical Soc., 2009.

\bibitem{Winkler12}
M.~Winkler.
\newblock Global large-data solutions in a {Chemotaxis-(Navier-)Stokes} system
  modeling cellular swimming in fluid drops.
\newblock {\em Communications in Partial Differential Equations Vol. 37 , Iss.
  2}, pages 319--351, 2012.

\bibitem{Winkler13}
M.~Winkler.
\newblock Finite-time blow-up in the higher-dimensional parabolic-parabolic
  {Keller-Segel} system.
\newblock {\em Journal de Mathématiques Pures et Appliquees, 100 (5)}, pages
  748--767, 2013.

\end{thebibliography}

\end{document}